\documentclass[11pt,reqno]{amsart}
\usepackage{amssymb,mathrsfs,graphicx,subfigure, enumerate}
\usepackage{amsmath,amsfonts,amssymb,amscd,amsthm,bbm}
\usepackage{extpfeil}
\usepackage{graphicx,colortbl}
\usepackage{float}
\usepackage{epsfig}
\usepackage{caption}
\usepackage{bm}
\graphicspath{{Figure/}}

\usepackage[margin=1in]{geometry}

\title[Large-time behavior of composite waves of viscous shocks for the NS equations]{Large-time behavior of composite waves of viscous shocks for the barotropic Navier-Stokes equations}

\author[Han]{Sungho Han}
\address[Sungho Han]{\newline Department of Mathematical Sciences \newline Korea Advanced Institute of Science and Technology, Daejeon  34141, Republic of Korea}
\email{sungho\_han@kaist.ac.kr}

\author[Kang]{Moon-Jin Kang}
\address[Moon-Jin Kang]{\newline Department of Mathematical Sciences \newline Korea Advanced Institute of Science and Technology, Daejeon  34141, Republic of Korea}
\email{moonjinkang@kaist.ac.kr}

\author[Kim]{Jeongho Kim}
\address[Jeongho Kim]{\newline School of Mathematics, \newline Korea Institute for Advanced Study, 85 Hoegiro, Seoul 02455, Republic of Korea}
\email{jeonghokim206@gmail.com}

\newtheorem{theorem}{Theorem}[section]
\newtheorem{lemma}{Lemma}[section]

\newtheorem{proposition}{Proposition}[section]
\newtheorem{remark}{Remark}[section]

\newcommand{\beq}{\begin{equation}}
\newcommand{\eeq}{\end{equation}}

\newcommand{\bbr}{\mathbb R}

\newcommand{\eps}{\varepsilon }

\newcommand{\R}{\mathbb{R}}

\newcommand{\bq}{\begin{equation}}
\newcommand{\eq}{\end{equation}}
\newcommand{\e}{\varepsilon}

\newcommand{\pa}{\partial}

\newcommand{\cB}{\mathcal{B}}
\newcommand{\cG}{\mathcal{G}}
\newcommand{\cS}{\mathcal{S}}
\newcommand{\cD}{\mathcal{D}}

\newcommand{\tU}{\widetilde{U}}
\newcommand{\tu}{\widetilde{u}}

\newcommand{\tv}{\widetilde{v}}
\newcommand{\tvi}{\widetilde{v}_i}

\newcommand{\pv}{p(v)}
\newcommand{\tp}{\widetilde{p}}
\newcommand{\tpv}{p(\widetilde{v})}

\newcommand{\thi}{\widetilde{h}_i}

\newcommand{\ai}{a_i}
\newcommand{\lpi}{\phi_i}

\newcommand{\norm}[1]{\left\lVert#1\right\rVert}
\newcommand{\di}{\displaystyle}

\begin{document}
\bibliographystyle{plain}

\date{\today}



\thanks{\textbf{Acknowledgment.} S. Han is supported by the NRF-2019R1C1C1009355.
M.-J. Kang is partially supported by the NRF-2019R1C1C1009355 and the POSCO Science Fellowship of POSCO TJ Park Foundation. 
J. Kim is supported by a KIAS Individual Grant (SP087401) via the Center for Mathematical Challenges at Korea Institute for Advanced Study.}

\begin{abstract} 
	We study the large-time behavior of the 1D barotropic Navier-Stokes flow perturbed from Riemann data generating a composition of two shock waves with small amplitudes. We prove that the perturbed Navier-Stokes flow converges, uniformly in space, towards a composition of two viscous shock waves  as time goes to infinity, up to dynamical shifts. Especially, the strengths of the two waves can be chosen independently. This is the first result for the convergence to a composite wave of two viscous shocks with independently small amplitudes.   
\end{abstract}

\maketitle

\tableofcontents

\section{Introduction}\label{sec:1}
\setcounter{equation}{0}

We consider the one-dimensional barotropic compressible Navier-Stokes equations in the Lagrangian mass coordinates:
\begin{align}
\begin{aligned}\label{eq:NS}
&v_t - u_x = 0,\\
&u_t +p(v)_x = \left(\mu\frac{u_x}{v}\right)_x, \quad x\in\bbr,\quad t>0
\end{aligned}
\end{align}
where $v=v(t,x)>0$ and $u=u(t,x)$ denote the specific volume and the velocity of the fluid, respectively. The pressure $p=p(v)$ is given by the $\gamma$-law $p(v) = bv^{-\gamma}$, with $b>0$ and $\gamma>1$ and $\mu>0$ denotes the viscosity coefficient of the fluid. For simplicity, we normalize the coefficients as $b=1$ and $\mu=1$. \\
Consider initial data of the system \eqref{eq:NS} given by $(v_0,u_0)$, which connects prescribed far-field constant states:
\begin{equation}\label{initial}
\lim_{x\to\pm\infty}(v_0(x),u_0(x)) = (v_{\pm},u_{\pm}).
\end{equation}

A heuristic argument (see e.g. \cite{MatsumuraBook}) describes that large-time behavior of solutions $(v,u)$ to the Navier-Stokes equations \eqref{eq:NS} has a close relationship with the Riemann problem of the associated Euler equations:
\begin{align}
\begin{aligned}\label{eq:Euler}
&v_t - u_x = 0,\\
&u_t + p(v)_x = 0,\quad x\in\bbr,\quad t>0,
\end{aligned}
\end{align}
subject to the Riemann initial data
\begin{equation}\label{Riemann_data}
(v(t,x),u(t,x))|_{t=0} = \begin{cases}
(v_-,u_-),\quad x<0,\\
(v_+,u_+),\quad x>0.
\end{cases}
\end{equation}
We consider the end states $(v_\pm,u_\pm)$ such that there exists a unique intermediate state $(v_m,u_m)$ which is connected with $(v_-,u_-)$ by 1-shock curve and with $(v_+,u_+)$ by 2-shock curve. That is, there exists a unique $(v_m,u_m)$ such that the following Rankine-Hugoniot condition and Lax entropy condition hold:
\begin{align}
\begin{aligned}\label{RH-condition}
&\begin{cases}
-\sigma_1(v_m-v_-)-(u_m-u_-) = 0,\\
-\sigma_1(u_m-u_-)+(p(v_m)-p(v_-))=0,
\end{cases}\quad \sigma_1 := -\sqrt{-\frac{p(v_m)-p(v_-)}{v_m-v_-}},\quad v_->v_m,\quad u_->u_m;\\
&\begin{cases}
-\sigma_2(v_+-v_m)-(u_+-u_m) = 0,\\
-\sigma_2(u_+-u_m)+(p(v_+)-p(v_m))=0,
\end{cases}\quad \sigma_2 := \sqrt{-\frac{p(v_+)-p(v_m)}{v_+-v_m}},\quad v_m<v_+,\quad u_m>u_+.
\end{aligned}
\end{align}
Then, the Euler equations \eqref{eq:Euler} with \eqref{Riemann_data}-\eqref{RH-condition} admit a unique self-similar solution, the so-called Riemann solution $(\bar{v},\bar{u})$, represented by the composition $(\bar{v},\bar{u})=(v^s_1,u^s_1)+(v^s_2,u^s_2)-(v_m,u_m)$ of 1-shock wave $(v^s_1,u^s_1)$ and 2-shock wave $(v^s_2,u^s_2)$ defined as (see e.g. \cite{Serre-book})
\[(v^s_1,u^s_1)(t,x) = \begin{cases}
(v_-,u_-),\quad x<\sigma_1t,\\
(v_m,u_m),\quad x>\sigma_1t,
\end{cases},\quad (v^s_2,u^s_2)(t,x) = \begin{cases}
(v_m,u_m),\quad x<\sigma_2t,\\
(v_+,u_+),\quad x>\sigma_2t.
\end{cases}\]
The viscous counterpart of the Riemann solution $(\bar{v},\bar{u})$ is given by the composite wave:
\beq\label{com2}
(\tv(t,x),\tu(t,x)):=\Big(\tv_1(x-\sigma_1 t),\tu_1(x-\sigma_1t)\Big)+\Big(\tv_2(x-\sigma_2 t),\tu_2(x-\sigma_2 t)\Big)-(v_m,u_m),
\eeq
which is composed of 1-viscous shock $(\tv_1,\tu_1)(x-\sigma_1t)$ and 2-viscous shock $(\tv_2,\tu_2)(x-\sigma_2t)$ satisfying: for each $i=1,2,$
\begin{equation}\label{viscous-shock-u}
\begin{cases}
-\sigma_i(\tv_i)' - (\tu_i)' = 0,\\
-\sigma_i(\tu_i)' + p(\tv_i)' = \left(\frac{(\tu_i)'}{\tv_i}\right)',\\
(\tv_1,\tu_1)(-\infty) = (v_-,u_-),\quad (\tv_1,\tu_1)(+\infty) = (v_m,u_m),\\
(\tv_2,\tu_2)(-\infty) = (v_m,u_m),\quad (\tv_2,\tu_2)(+\infty) = (v_+,u_+).
\end{cases}
\end{equation} 
Notice that each viscous shock $(\tv_i(x-\sigma_i t),\tu_i(x-\sigma_it))$ is a traveling wave solution to \eqref{eq:NS}.\\

In this paper, we aims to prove that solutions to the Navier-Stokes system \eqref{eq:NS}-\eqref{initial} with \eqref{RH-condition} converge to the composite wave $(\bar{v},\bar{u})$ up to shifts, uniformly in $x$ as $t\to\infty$. \\

As previous results on time-asymptotic stability for \eqref{eq:NS}-\eqref{initial} when the end states are connected by a single shock, Matsumura and Nishihara \cite{MN-S} first proved the convergence of solutions toward a single viscous shock with small amplitude, uniformly in $x$, using the anti-derivative method under the zero-mass condition. Later on, this zero-mass condition is removed by introducing a constant shift with diffusion wave by Liu \cite{Liu}, Liu-Zeng \cite{LZ} and Szepessy-Xin \cite{SX}. On the other hand, Masica and Zumbrun \cite{M-Zumbrun} showed the spectral stability of viscous shock wave under the weaker condition compared to the zero-mass condition, called a spectral condition. 
Recently, Wang-Wang \cite{WW} studied on a planar shock wave for the three-dimensional barotropic Navier-Stokes equations,  by utilizing a new method called ``$a$-contraction with shifts".

For the composition of two shocks as in our setting, Huang-Matsumura \cite{HM} showed the convergence toward a composite wave composed of two viscous shocks for the Navier-Stokes-Fourier system when the strengths of two viscous shocks are small with the same order. As mentioned in \cite{MatsumuraBook}, the same result also could be obtained for the barotropic case \eqref{eq:NS}, using the parallel argument as in \cite{HM,MN-S}. \\
In this paper, however, we do not assume the same order smallness, that is, our result provides the uniform convergence toward a composition of two viscous shocks with independently small amplitudes. 

On the other hand, when the initial data \eqref{Riemann_data} generates rarefaction waves, the time-asymptotic stability of the rarefaction wave has been proven by Matsumura-Nishihara \cite{MN86,MN}. Similar stability results are also shown for the Navier-Stokes-Fourier system in \cite{LX,NYZ}.

However, all the mentioned literature treated either shocks or rarefaction waves, but not the composition of them. Indeed, the time-asymptotic stability of the composition of shock and rarefaction waves is a challenging problem \cite{MN-S}. This open problem is recently solved by the second author, Vasseur and Wang \cite{KVW3}, using the method of $a$-contraction with shift for the shock wave, combining with the energy method for the rarefaction wave.\\ 

We will apply the method of $a$-contraction with shifts to prove the time-asymptotic behavior as a composition of two viscous shock waves in the following theorem.

\begin{theorem}\label{thm:main}
	For a given constant state $(v_+,u_+)\in\R_+\times\R$, there exist positive constants $\delta_0,\e_0$ such that the following holds.\\	
	For any constant states $(v_m,u_m)$ and $(v_-,u_-)$ satisfying \eqref{RH-condition} with
	\beq\label{assind}
	|v_+-v_m|+|v_m-v_-|<\delta_0,
	\eeq
	let $(\tv_i,\tu_i)(x-\sigma_it)$ be the $i$-viscous shock wave satisfying \eqref{viscous-shock-u}. In addition, let $(v_0,u_0)$ be any initial data satisfying
	\[\sum_{\pm} \left(\|v_0-v_\pm\|_{L^2(\R_\pm)}+\|u_0-u_\pm\|_{L^2(\R_\pm)}\right)+ \|\pa_xv_0\|_{L^2(\R)}+\|\pa_xu_0\|_{L^2(\R)}<\e_0,\]
	where $\R_+:=(0,+\infty)$ and $\R_-:=(-\infty,0)$. Then, the compressible Navier-Stokes system \eqref{eq:NS}-\eqref{initial} with \eqref{RH-condition} admits a unique global-in-time solution $(v,u)$ in the following sense: there exist absolutely continuous shift functions $X_1(t), X_2(t)$ such that
	\begin{align*}
	&v(t,x)-(\tv_1(x-\sigma_1t-X_1(t))+\tv_2(x-\sigma_2t-X_2(t))-v_m)\in C(0,+\infty;H^1(\R)),\\
	&u(t,x)-(\tu_1(x-\sigma_1t-X_2(t))+\tu_2(x-\sigma_2t-X_2(t))-u_m)\in C(0,+\infty;H^1(\R)).
	\end{align*}
	Moreover, we have the large-time behavior:
	\begin{align*}
	&\lim_{t\to+\infty}\sup_{x\in \R}\Big|v(t,x)-\big(\tv_1(x-\sigma_1t-X_1(t))+\tv_2(x-\sigma_2t-X_2(t))-v_m\big)\Big|=0,\\
	&\lim_{t\to+\infty}\sup_{x\in \R}\Big|u(t,x)-\big(\tu_1(x-\sigma_1t-X_1(t))+\tu_2(x-\sigma_2t-X_2(t))-u_m\big)\Big|=0,
	\end{align*}
	where 
	\beq\label{as}
	\lim_{t\to+\infty}|\dot{X}_i(t)|=0,\quad\mbox{for}\quad i=1,2.
	\eeq
	Especially, the shifts are well-separated in the following sense: 
	\beq\label{sx12}
	X_1(t)+\sigma_1t \le \frac{\sigma_1}{2}t<0<\frac{\sigma_2}{2}t \le X_2(t) +\sigma_2t ,\quad t>0.
	\eeq	
\end{theorem}

\begin{remark}
	1. It follows from the Rankine-Hugoniot condition \eqref{RH-condition} that the strength of 1-shock (resp. 2-shock) can be measured by $|v_--v_m|$ (resp. $|v_m-v_+|$). Thus, by the smallness condition \eqref{assind}, the strengths of the two shocks can be chosen independently. This removes the same order assumption, i.e. $|v_--v_m| + |v_--v_m| \lesssim \min(|v_--v_m|, |v_--v_m|)\ll1$ that is a crucial assumption in \cite{HM}.\\
	2. 
	The property \eqref{as} implies that 
	\[
	\lim_{t\to+\infty}\frac{X_i(t)}{t}=0,\quad\mbox{for}\quad i=1,2.
	\]
	Thus, the shift functions $X_i(t)$ grow at most sub-linearly w.r.t. $t$, by which the shifted composite wave 
	\[
	\big(\tv_1, \tu_1\big) \big(x-\sigma_1t-X_1(t)\big)+\big(\tv_2, \tu_2\big) \big(x-\sigma_2t-X_2(t)\big) - (v_m,u_m)
	\]
	time-asymptotically keeps the original composite wave.
\end{remark}

The remaining part of the paper is organized as follows. In Section \ref{sec:2}, we present a main idea of the proof, and useful estimates for the shock waves. Elementary estimates for the pressure and relative entropy are also presented. In Section \ref{sec:3}, we present the proof of the main theorem under the a priori estimates. The a priori estimates will be proved in the following two sections. In Section \ref{sec:4}, we use the method of $a$-contraction with shifts to estimate the $L^2$-perturbation. Then, the remaining estimates will be obtained  in Section \ref{sec:5}.

\section{Main ideas of the proof and useful lemmas}\label{sec:2}
\setcounter{equation}{0}
In this section, we present a main idea of the proof, and useful estimates for the shock waves and relative quantities. 

\subsection{Main ideas and the method of $a$-contraction with shifts}
The main tool to obtain the desired nonlinear stability is the method of $a$-contraction with shifts, which was introduced by the second author and Vasseur in \cite{KVARMA} (see also \cite{Vasseur-2013}) to study the stability of extremal shocks for the hyperbolic system of conservation laws such as the Euler system \eqref{eq:Euler}.  

This method has been extended to studying viscous models as follows. For the one-dimensional barotropic Navier-Stokes system, the contraction property of any large perturbations for a single viscous shock is proved in \cite{Kang-V-NS17}, and for a composite wave of two shocks in \cite{KV-2shock}.  Those results play a crucial role to prove the stability of entropy Riemann shocks of the isentropic Euler system in the class of inviscid limits from the Navier-Stokes system in \cite{Kang-V-NS17,KV-Inven}. As mentioned in Introduction, this method was also used in \cite{KVW3,WW} to show the long-time behavior of the barotropic Navier-Stokes system for the composition of shock and rarefaction in 1D, and for a single shock in multi-D. 
As applications of the method to other viscous hyperbolic systems, we also refer to \cite{Kang19,Kang-V-1,KVW} for viscous scalar conservation laws, \cite{KVW2} for the stability of a planar contact discontinuity of the 3D full Euler system in the class of zero dissipation limits,  and \cite{CKKV,CKV} for the viscous hyperbolic system arising from a chemotaxis model. \\

To illustrate a key idea of the method for the viscous system \eqref{eq:NS}, consider the entropy $\eta$ of the Euler system \eqref{eq:Euler} defined as 
\[\eta(U):=\frac{u^2}{2}+Q(v),\quad Q(v):=\frac{1}{(\gamma-1)v^{\gamma-1}},\]
where $U=(v,u)$. Then, the relative entropy $\eta(U|\overline{U})$ between two states $U$ and $\overline{U}$ is defined as
\[\eta(U|\overline{U}):=\eta(U)-\eta(\overline{U})-D\eta(\overline{U})(U-\overline{U})=\frac{|u-\overline{u}|^2}{2}+Q(v|\overline{v}),\]
where $Q(v|\overline{v}):=Q(v)-Q(\overline{v})-Q'(\overline{v})(v-\overline{v})$. Since $Q(v)$ is a strictly convex function in $v$, $Q(\cdot|\cdot)$ is locally quadratic in the sense that for any $0<a<b$, there exists $C>0$ such that
\[
C^{-1}|v_1-v_2|^2\le Q(v_1|v_2) \le C |v_1-v_2|^2,\qquad \forall v_1, v_2 \in [a,b].
\]

When $\overline{U}(x-\sigma t)$ is a single viscous shock, the method is to find a weight function $a:\bbr\to\bbr_+$ and a time-dependent shift function $X:\bbr_+\to\bbr$ such that the weighted relative entropy with shift is not increasing in time:
$$
\frac{d}{dt}\int_\R a(x-\sigma t -{X}(t))\eta(U(t,x)|\overline{U}(x-\sigma t - {X}(t)))\,dx \leq 0.
$$
First, by a standard computation based on the relative entropy method, the left-hand side can be decomposed into three parts (as in Lemma \ref{lem:req} for example):
\[
\mbox{LHS}= X'(t) Y(U) + \mathcal{J}^{\rm bad} (U)-\mathcal{J}^{\rm good} (U),
\]
where $ \mathcal{J}^{\rm bad} (U)$ and $\mathcal{J}^{\rm good} (U)$ consist of the all bad terms and all good terms respectively.
To make the right-hand side non-positive, we might use the typical energy method for parabolic equations. However, since the barotropic Naiver-Stokes system  has the dissipation in one variable only (more precisely, in the $u$ variable for \eqref{eq:NS}), the weight function $a$ would be found to provide an additional good term in terms of the $v$ variable, by which the bad terms could be represented only by the $u$ variables. Indeed, since $\sigma$ is a non-zero constant,
constructing a monotone function $a$ satisfying $\sigma a'>0$, we have a good term 
\[
-\sigma \int_\R a'(x-\sigma t -{X}(t))\eta(U(t,x)|\overline{U}(x-\sigma t - {X}(t)))\,dx.
\]
In fact, the weight function $a$ will be defined by the first component $\overline{v}$ of the viscous shock such that $a'$ localizes the perturbation in space as done by $\overline{v}'$, and the image of $a$ is a bounded open interval.  
Using the above term, we maximize all bad terms in terms of the $v$ variable, from which the remaining bad terms are related to the $u$ variable only, and localized by $a'$ or $\overline{v}'$.
To absorb the remaining bad terms by the diffusion term, we may use the following Poincar\'e-type inequality. 
\begin{lemma} \cite[Lemma 2.9]{Kang-V-NS17}\label{lem-poin}
	For any $f:[0,1]\to\bbr$ satisfying $\int_0^1 y(1-y)|f'|^2 dy<\infty$, 
	\beq\label{poincare}
	\int_0^1\left|f-\int_0^1 f dy \right|^2 dy\le \frac{1}{2}\int_0^1 y(1-y)|f'|^2 dy.
	\eeq
\end{lemma}
However, to apply Lemma \ref{lem-poin}, we may need the other good term as an average of linear perturbation on $u$-variable, which together with the bad terms would give a variance of the perturbation that could be absorbed by the dissipation as in the Poincar\'e-type inequality. Here, the desired good term would be extracted from the shift part $X'(t) Y(U)$ by a sophisticated construction of the shift $X(t)$. This gives the desired contraction estimate.\\

However, the composition \eqref{com2} of two viscous shocks leads to several problematic issues. Since the composition of two viscous shocks is not an exact solution to the Navier-Stokes equations, unlike the case of a single viscous shock, due to the nonlinear terms in the equations, we need to control some error terms on the nonlinear interactions between two shock waves. Second, as we consider two shock waves, we should construct two shift functions, say $X_1$ and $X_2$, and two weight functions to control perturbation near each shock wave, based on the above method for a single wave. On top of that, since all the bad terms are localized by the derivative of each weight or shock, whereas the diffusion term is not localized, we need to localize the diffusion term by introducing some auxiliary localization functions. The localization functions will be defined by two shift functions as in \eqref{localization}, by constructing the two shift functions to be well-separated like \eqref{sx12}.

\subsection{Viscous shock waves}
We review the useful estimates of viscous shocks and their derivatives that will be used throughout the paper.
\begin{lemma}\cite[Lemma 5.1]{KV-2shock}\label{lem:shock-est}
	For a given constant $U_*:=(v_*,u_*)\in\bbr_+\times\bbr$, there exist positive constants $\delta_0$, $C$, $C_1$ and $C_2$ such that the following holds. Let $U_-:=(v_-,u_-)$, $U_m:=(v_m,u_m)$, and $U_+:=(v_+,u_+)\in\bbr_+\times\bbr$ be any constants such that $U_-,U_m,U_+\in B_{\delta_0}(U_*)$, and $|p(v_-)-p(v_m)|=:\delta_1<\delta_0$ and $|p(v_m)-p(v_+)|=:\delta_2<\delta_0$. Let $\tU_1=(\tv_1,\tu_1)$ and $\tU_2=(\tv_2,\tu_2)$ be the 1- and 2-shocks connecting from $U_-$ to $U_m$ and from $U_m$ to $U_+$ respectively, satisfying $\tv_1(0)=\frac{v_-+v_m}{2}$ and $\tv_2(0)=\frac{v_m+v_+}{2}$ without loss of generality. Then, the following estimates hold: for each $i=1,2$,
	\begin{align}
	\begin{aligned}\label{shock-est}
	& v_i' \sim u_i'\quad\mbox{i.e.,}\quad C^{-1} \tv_i'(x-\sigma_it) \le \tu_i'(x-\sigma_it) \le C\tv_i'(x-\sigma_it),\quad x\in\bbr,\, t>0,\\
	&C^{-1}\delta_ie^{-C_1\delta_i|x-\sigma_it|}\le \tv_i(x-\sigma_it)-v_m\le C\delta_ie^{-C_2\delta_i|x-\sigma_it|},\quad (-1)^i(x- \sigma_i t)\le 0,\\
	&-C^{-1}\delta_i^2e^{-C_1\delta_i|x-\sigma_it|}\le  \tv_i'(x-\sigma_it)\le -C\delta_i^2e^{-C_2\delta_i|x-\sigma_it|},\quad x\in\bbr,\quad t>0,\\
	\end{aligned}		
	\end{align}
	in addition,
	\begin{equation}\label{second-order}
	|(\tv_i''(x-\sigma_it),\tu_i''(x-\sigma_it))|\le C\delta_i|(\tv_i'(x-\sigma_it),\tu_i'(x-\sigma_it))|.
	\end{equation}
\end{lemma}

\begin{remark}
	Throughout the paper, we will use Lemma \ref{lem:shock-est} with $U_*=U_+$. Thus, the constants $\delta_0$, $C$, $C_1$ and $C_2$ depend only on $U_+$. 
\end{remark}

\subsection{Estimate on the relative quantities}
As our method is based on the estimate of relative entropy, we need to control the relative quantities, such as the relative pressure $p(v|w)$ or the relative internal energy $Q(v|w)$. Recall that the relative quantities are defined as
\[p(v|w) = p(v)-p(w)-p'(w)(v-w),\quad Q(v|w) = Q(v)-Q(w)-Q'(w)(v-w).\]
Considering the Taylor expansion, it is expected that they are almost quadratic quantities at least locally. The following estimates show the exact estimate on this locally quadratic behavior of the relative quantities. The proof of the lemma can be found in \cite{Kang-V-NS17}.

\begin{lemma}\label{lem-rel-quant}
	Let $\gamma>1$ and $v_+$ be given constants. Then, there exist constants $C,\delta_*>0$ such that the following assertions hold:
	\begin{enumerate}
		\item For any $v,w$ satisfying $0<w<2v_+$ and $0<v<3v_+$,
		\begin{equation}\label{est-rel-1}
		|v-w|^2\le CQ(v|w),\quad |v-w|^2\le Cp(v|w).
		\end{equation}
		\item For any $v,w$ satisfying $v,w>v_+/2$,
		\begin{equation}\label{est-rel-2}
		|p(v)-p(w)|\le C|v-w|.
		\end{equation}
		\item For any $0<\delta<\delta_*$ and any $(v,w)\in\bbr_+^2$ satisfying $|p(v)-p(w)|<\delta$ and $|p(w)-p(v_+)|<\delta$,
		\begin{align}
		\begin{aligned}\label{est-rel-3}
		&p(v|w) \le \left(\frac{\gamma+1}{2\gamma}\frac{1}{p(w)}+C\delta\right)|p(v)-p(w)|^2,\\
		&Q(v|w) \ge \frac{p(w)^{-\frac{1}{\gamma}-1}}{2\gamma}|p(v)-p(w)|^2-\frac{1+\gamma}{3\gamma^2}p(w)^{-\frac{1}{\gamma}-2}(p(v)-p(w))^3,\\ 
		&Q(v|w)\le \left(\frac{p(w)^{-\frac{1}{\gamma}-1}}{2\gamma}+C\delta\right)|p(v)-p(w)|^2.
		\end{aligned}
		\end{align}
	\end{enumerate}
\end{lemma}

\section{Proof of Theorem \ref{thm:main}}\label{sec:3}
\setcounter{equation}{0}
In this section, we present a main part of proof for Theorem \ref{thm:main}. 

\subsection{Local existence}
In order to get the desired result on time-asymptotic stability of $H^1$-perturbations of a composite wave connecting two different constant states  at far-fields, we need to recall the well-known result on local-in-time existence of solutions to \eqref{eq:NS} connecting the two different constant states in $H^1$-norm.  
\begin{proposition}\cite{MV-sima,Solo}\label{prop:local-H1}
Let $\underline{v}$ and $\underline{u}$ be smooth monotone functions such that
\[\underline{v}(x) = v_\pm,\quad \underline{u}(x) = u_\pm,\quad \mbox{for}\quad \pm x\ge 1.\]
Then, for any constants $M_0$, $M_1$, $\underline{\kappa}_0$, $\overline{\kappa}_0$, $\underline{\kappa}_1$, $\overline{\kappa}_1$ with 
\[0<M_0<M_1,\quad\mbox{and}\quad 0< \underline{\kappa}_1<\underline{\kappa}_0 < \overline{\kappa}_0 < \overline{\kappa}_1,\]
there exists a constant $T_0>0$ such that if the initial data $(v_0,u_0)$ satisfy
\[\|v_0-\underline{v}\|_{H^1(\bbr)}+\|u_0-\underline{u}\|_{H^1(\bbr)}\le M_0,\quad \mbox{and}\quad  \underline{\kappa}_0\le v_0(x)\le \overline{\kappa}_0,\quad \forall x\in\bbr,\]
then the Navier-Stokes equations \eqref{eq:NS} admit a unique solution $(v,u)$ on $[0,T_0]$ satisfying
\begin{align*}
	&v-\underline{v}\in C([0,T_0];H^1(\bbr)),\\
	&u-\underline{u}\in C([0,T_0];H^1(\bbr))\cap L^2([0,T_0];H^2(\bbr)),
\end{align*}
together with
\[
\|v-\underline{v}\|_{L^\infty([0,T_0];H^1(\bbr))}+\|u-\underline{u}\|_{L^\infty([0,T_0];H^1(\bbr))}\le M_1,
\]
and
\beq\label{bddab}
\underline{\kappa}_1\le v(t,x)\le \overline{\kappa}_1,\quad \forall (t,x)\in[0,T_0]\times\bbr.
\eeq
\end{proposition}

\subsection{Construction of shifts}
As desired, we should show that a small $H^1$-perturbation of a composite wave of two viscous shocks is orbitally stable for a large time, more precisely, the perturbation uniformly converges to the composite wave up to shifts where each shock is shifted by $X_i(t)$:
\begin{align}
	\begin{aligned}\label{composite_wave}
	&(\tv^{X_1,X_2},\tu^{X_1,X_2})(t,x) \\
	&\qquad := \left(\tv_1^{X_1}(x-\sigma_1t)+\tv_2^{X_2}(x-\sigma_2t)-v_m,\tu_1^{X_1}(x-\sigma_1t)+\tu_2^{X_2}(x-\sigma_2t)-u_m\right),
	\end{aligned}
\end{align}
where $F^{X_i}$ denotes a function $F$ shifted by $X_i$, that is, $F^{X_i}(x) := F(x-X_i(t))$ for any function $F$. This notation will be used throughout the paper. \\
We here introduce the shift functions explicitly, from which we could obtain a bound of derivative of shifts (at least locally-in-time) in Lemma \ref{lem:xex}, and obtain the desired a priori estimates  in Proposition \ref{prop:H1-estimate}. Those will be used for the continuation argument in Section \ref{subsec:global}.
We define a pair of shifts $(X_1, X_2)$ as a solution to the system of ODEs:
\begin{equation}\label{X(t)}
\left\{
\begin{array}{ll}
\di \dot{X}_1(t)=-\frac{M}{\delta_1}\bigg[
 \int_{\bbr}\frac{a^{X_1,X_2}}{\sigma_1}(\widetilde{h}^{X_1}_1)_x(p(v)-p(\widetilde{v}^{X_1,X_2}))\,dx  -\int_{\bbr}a^{X_1,X_2} (p(\widetilde{v}^{X_1}_1))_x(v-\widetilde{v}^{X_1,X_2})\,dx\bigg],\\ [4mm]
 \di \dot{X}_2(t)=-\frac{M}{\delta_2}\bigg[
 \int_{\bbr}\frac{a^{X_1,X_2}}{\sigma_2}(\widetilde{h}^{X_2}_2)_x(p(v)-p(\widetilde{v}^{X_1,X_2}))\,dx  -\int_{\bbr}a^{X_1,X_2} (p(\widetilde{v}^{X_2}_2))_x(v-\widetilde{v}^{X_1,X_2})\,dx\bigg],\\ [4mm]
\di  X_1(0)=X_2(0)=0,
\end{array}
\right.
\end{equation}
where $a^{X_1,X_2}$ is the shifted weight function defined in \eqref{weight}, $\widetilde{h}_i:=\widetilde{u}_i-(\ln \widetilde v_i)_x$, and $M$ is the specific constant chosen as $M:=\frac{5(\gamma+1)}{8\gamma p(v_m)} \big(-p'(v_m)\big)^{\frac32}$, which will be used in the proof of Proposition \ref{prop:H1-estimate}.
However, for well-posedness of the above ODEs in Lemma \ref{lem:xex}, we only need the following assumption for the shifted weight function at this point: $a^{X_1,X_2}$ is a $C^1$-function of $(x, X_1, X_2)$ with finite $C^1$-norm. This is verified by the explicit one defined in \eqref{weight}.\\
The following lemma ensures that \eqref{X(t)} has a unique absolutely continuous solution at least for the lifespan $[0,T_0]$ of solution $v$ satisfying \eqref{bddab}.
\begin{lemma}\label{lem:xex}
For any $c_1,c_2>0$, there exists a constant $C>0$ such that the following is true.  For any $T>0$, and any   function $v\in L^\infty ((0,T)\times \R)$ 
verifying
\beq\label{odes}
c_1 \le v(t,x)\le c_2,\qquad \forall (t,x)\in [0,T]\times \bbr,
\eeq
 the system \eqref{X(t)} has a unique absolutely continuous solution $(X_1,X_2)$ on $[0,T]$. Moreover, 
\beq\label{roughx}
|X_1(t)| + |X_2(t)|  \le Ct,\quad \forall t\le T.
\eeq
\end{lemma}
\begin{proof}
Let $F(t,X_1,X_2)$ denote the right-hand side of the system \eqref{X(t)}. \\
Using \eqref{odes} and the facts that the $C^1$-norm of $a^{X_1,X_2}$ w.r.t. $(x,X_1,X_2)$ is finite, and
\[
\|\widetilde h_i\|_{C^2} <\infty,\quad\|\widetilde v_i\|_{C^2} <\infty,\quad  \|(\widetilde h_i)_x\|_{L^1} \le C\delta_i,\quad \|(\widetilde v_i)_x\|_{L^1} \le C\delta_i,
\]
we have
\begin{align}
		\begin{aligned}\label{temco}
&\sup_{X_1,X_2\in\bbr}|F(t,X_1,X_2)|\\
&\quad \le C\|a\|_{C^1} \sum_{i=1}^2  \frac{1}{\delta_i} \Big(\|(p(\widetilde v_i)\|_{L^\infty} +\|p(v)\|_{L^\infty}+\|\widetilde v_i\|_{L^\infty}+\|v\|_{L^\infty}\Big)  (\|(\widetilde h_i)_x\|_{L^1} + \|(\widetilde v_i)_x\|_{L^1} ) \le C,
\end{aligned}
	\end{align}
where the constant $C$ is independent of $t$.
Likewise, we have
\[
\sup_{X_1,X_2\in\bbr} |\nabla_{X_1,X_2} F(t,X_1,X_2)|\le C.
\]
Therefore,  \eqref{X(t)}  has a unique absolutely continuous solution by a simple adaptation of Cauchy-Lipschitz theorem (for example, by  \cite[Lemma A.1]{CKKV} and  \cite[Lemma C.1]{KV-2shock}).  \\
Since $|\dot{X_1}(t)| + |\dot{X_2}(t)| \le C$ by \eqref{temco}, we have \eqref{roughx}.
\end{proof}

\subsection{Main proposition for a priori estimates}

Based on the previous setting, we present the main proposition for a priori estimate.

\begin{proposition}\label{prop:H1-estimate}
	For a given constant $U_+:=(v_+,u_+)\in\bbr_+\times\bbr$, there exist positive constants $\delta_0$, $C_0$, and $\e_1$ such that the following holds: \\ 
	For any constant states $U_m:=(v_m,u_m)\in S_2(v_+,u_+)$ and $U_+:=(v_-,u_-)\in S_1(v_m,u_m)$ satisfying
	$|p(v_-)-p(v_m)|=:\delta_1<\delta_0$ and $|p(v_m)-p(v_+)|=:\delta_2<\delta_0$, let $(\tv^{X_1,X_2},\tu^{X_1,X_2})$ denote the composite wave of two shifted shocks as in \eqref{composite_wave}, where $(X_1,X_2)$ solves \eqref{X(t)}. 
	Suppose that $(v,u)$ is the solution to $\eqref{eq:NS}$ on $[0,T]$ for some $T>0$,  and satisfy  
		\begin{align*}
		&v-\tv^{X_1,X_2} \in C([0,T];H^1(\bbr)),\\
		&u-\tu^{X_1,X_2} \in C([0,T];H^1(\bbr))\cap L^2(0,T;H^2(\bbr)),
	\end{align*} 
	and
	\begin{equation}\label{perturbation_small}
		\|v-\tv^{X_1,X_2}\|_{L^\infty(0,T;H^1(\bbr))} + \|u-\tu^{X_1,X_2}\|_{L^\infty(0,T;H^1(\bbr))}\le \e_1.
	\end{equation}
	Then, for all $t\in [0,T]$,
	
	\begin{align}
		\begin{aligned}\label{C-3}
			\sup_{t \in [0,T]}&\left(\|v-\tv^{X_1,X_2}\|_{H^1(\bbr)}+\|u-\tu^{X_1,X_2}\|_{H^1(\bbr)}\right)+\sqrt{\int_0^t \sum_{i=1}^2\delta_i |X_i|^2\,ds}\\
			&\quad +\sqrt{\int_0^t \left(  \mathcal{G}^S(U)+{D}(U)+D_1(U)+D_2(U) \right)\,ds}\\
			&\le C_0\left(\|v_0-\tv(0,\cdot)\|_{H^1(\R)}+\|u_0-\tu(0,\cdot)\|_{H^1(\R)}\right)+C_0 \delta_0^{1/4}.
		\end{aligned}
	\end{align}
	In particular,
\beq\label{bddx12}
|\dot{X}_1(t)| + |\dot{X}_2(t)|\le C_0 \|(v-\tv^{X_1,X_2})(t,\cdot)\|_{L^\infty(\bbr)},\qquad t\le T.
\eeq
	Here, the constant $C_0$  is independent of $T$ and
	\begin{equation} \label{tempo}
	\begin{aligned}
	\mathcal{G}^S(U)&:=\sum_{i=1}^2\int_\R |(\tv_i)^{X_i}_x| |\phi_i(v-\tv^{X_1,X_2})|^2 \, dx,\\
	D(U)&:=\int_\R | \partial_x (p(v)-p(\tv^{X_1,X_2}))|^2 \,dx,\\
	D_1(U)&:=\int_\R |(u-\tu^{X_1,X_2})_x|^2 \, dx,\\
	D_2(U)&:=\int_\R |(u-\tu^{X_1,X_2})_{xx}|^2 \, dx,
	\end{aligned}
	\end{equation}
where $\phi_1, \phi_2$ are {\it cutoff} functions  defined by
\begin{equation}\label{localization}
	\phi_1(t,x):=\begin{cases}
	1&\quad\mbox{if}\quad x<\frac{X_1(t)+\sigma_1t}{2},\\
	0&\quad\mbox{if}\quad x>\frac{X_2(t)+\sigma_2t}{2},\\
	\mbox{linearly decreasing $1$ to $0$} &\quad\mbox{if}\quad \frac{X_1+\sigma_1 t}{2}\le x\le \frac{X_2+\sigma_2t}{2},
\end{cases}\quad \phi_2(t,x):=1-\phi_1(t,x).
\end{equation}	
\end{proposition}

\begin{remark}\label{rem-phi}
Since the (shifted) composite wave $(\tv^{X_1,X_2}, \tu^{X_1,X_2})$ is not a solution to the Navier-Stokes equation as in \eqref{com-eqn}, we need to control the interaction term of the 1- and 2-waves to get the desired results of Proposition \ref{prop:H1-estimate}. For that, we localize the perturbation near each wave by using the above cutoff functions. More specifically, $\phi_1$ (resp. $\phi_2$) localizes the perturbation near the 1-wave (resp. 2-wave) shifted by $X_1$ (resp. $X_2$). Notice from  \eqref{sepx12} that 
\[
 X_2(t) +\sigma_2t \ge \frac{\sigma_2}{2}t >0> \frac{\sigma_1}{2}t\ge  X_1(t) +\sigma_1t,\qquad t>0,
\]
and so the functions $\phi_1$ and $\phi_2$ are well-separated as time goes.   
\end{remark}

\subsection{Global-in-time existence for perturbation} \label{subsec:global}

We here prove the global existence of solution by using the continuation argument based on Proposition \ref{prop:local-H1} and Proposition \ref{prop:H1-estimate}. We first assume that the constants $\delta_0,C$, and $\e_1$ are given by Proposition \ref{prop:H1-estimate}, and consider the end states $(v_-,u_-)$ and $(v_+,u_+)$ satisfying the conditions in Proposition \ref{prop:H1-estimate}. First of all, the smooth monotone functions $\underline{v}(x)$ and $\underline{u}(x)$ given in Proposition \ref{prop:local-H1} satisfy
\begin{align}
\begin{aligned}\label{monotone-est}
	\sum_{\pm}&\left(\|\underline{v}-v_{\pm}\|_{L^2(\bbr_\pm)}+\|\underline{u}-u_\pm\|_{L^2(\bbr_\pm)}\right)+\|\pa_x\underline{v}\|_{L^2(\bbr)}+\|\pa_x\underline{u}\|_{L^2(\bbr)}\\
	&\le C(|v_+-v_-|+|u_+-u_-|)\le C_1(\delta_1+\delta_2).
\end{aligned}
\end{align}
Then, we estimate the $H^1$-perturbation between the monotone functions $\underline{v}, \underline{u}$ and the initial composite waves $\tv(0,\cdot),\tu(0,\cdot)$ in \eqref{com2} as

\begin{align}
	\begin{aligned}\label{H1-perterb-1}
	&\|\underline{v}(\cdot)-\tv(0,\cdot)\|_{H^1(\R)}+\|\underline{u}(\cdot)-\tu(0,\cdot)\|_{H^1(\R)}\\
	&\le \|\underline{v}-\tv(0,\cdot)\|_{L^2(\R)}+\|\pa_x(\underline{v}-\tv(0,\cdot))\|_{L^2(\R)}+\|\underline{u}-\tu(0,\cdot)\|_{L^2(\R)}+\|\pa_x(\underline{u}-\tu(0,\cdot))\|_{L^2(\R)}\\
	&\le \sum_{\pm}\left(\|\underline{v}-v_\pm\|_{L^2(\R_\pm)}+\|\underline{u}-u_\pm\|_{L^2(\R_\pm)}\right)\\
	&\quad + \|\tv_1-v_m\|_{L^2(\R_+)} + \|\tv_1-v_-\|_{L^2(\R_-)} +\|\tv_2-v_+\|_{L^2(\R_+)}+\|\tv_2-v_m\|_{L^2(\R_-)}\\
	&\quad + \|\pa_x\underline{v}\|_{L^2(\R)} + \|\pa_x \tv_1\|_{L^2(\R)}+\|\pa_x\tv_2\|_{L^2(\R)}\\
	&\quad + \|\tu_1-u_m\|_{L^2(\R_+)} + \|\tu_1-u_-\|_{L^2(\R_-)} +\|\tu_2-u_+\|_{L^2(\R_+)}+\|\tu_2-u_m\|_{L^2(\R_-)}\\
	&\quad + \|\pa_x\underline{u}\|_{L^2(\R)} + \|\pa_x \tu_1\|_{L^2(\R)}+\|\pa_x\tu_2\|_{L^2(\R)}\\
	&\le C_2(\sqrt{\delta_1}+\sqrt{\delta_2}),
	\end{aligned}
\end{align}
where we used the shock estimates \eqref{shock-est} and \eqref{monotone-est} in the last inequality. Now, for the constant $\e_1$, we choose $\delta_0$ small enough so that the following condition holds for any $\delta_1, \delta_2<\delta_0$:
\[\frac{\frac{\e_1}{4}-C_0\delta_0^{1/4}}{C_0+1}-C_1(\delta_1+\delta_2)-C_2(\sqrt{\delta_1}+\sqrt{\delta_2})\ge 0.
\]
Consider the two positive constants $\e_*$ and $\e_0$ defined as
\[
\e_*:=\frac{\frac{\e_1}{2}-C_0\delta_0^{1/4}}{C_0+1}-C_2(\sqrt{\delta_1}+\sqrt{\delta_2}),\qquad
\e_0:=  \frac{\e_1}{4(C_0+1)}.
\]
We now consider any initial data $(v_0,u_0)$ satisfying
\[\sum_\pm (\|v_0-v_\pm\|_{L^2(\R_\pm)}+\|u_0-u_\pm\|_{L^2(\R_\pm)})+\|\pa_xv_0\|_{L^2(\R)} +\|\pa_xu_0\|_{L^2(\R)}<\e_0,\]
which implies that the initial $H^1$-perturbation from the monotone functions $\underline{v},\underline{u}$ is small:
\begin{equation}\label{H1-perturb-initial}
	\|v_0-\underline{v}\|_{H^1(\R)}+\|u_0-\underline{u}\|_{H^1(\R)}\le \e_0+C_1(\delta_1+\delta_2)\le \e_*.
\end{equation}
In particular, Sobolev embedding implies $\|v_0-\underline{v}\|_{L^\infty(\R)}\le C\e_*$, and therefore,
\[\frac{v_-}{2}<v_0(x)<2v_+,\quad \forall x\in \bbr.\]
Since $0<\e_*<\frac{\e_1}{2}$, we use Proposition \ref{prop:local-H1} to guarantee the local existence of the unique solution $(v,u)$ on $[0,T_0]$ satisfying
\begin{equation}\label{C-6}
	\|v-\underline{v}\|_{L^\infty(0,T_0;H^1(\R))}+\|u-\underline{u}\|_{L^\infty(0,T_0;H^1(\R))}\le \frac{\e_1}{2},
\end{equation}
and 
\[\frac{v_-}{3}<v(t,x)<3v_+,\quad \forall (t,x)\in[0,T_0]\times\R.\]
We now use a similar argument as in \eqref{H1-perterb-1} to obtain 
\begin{align*}
	&\|\underline{v}-\tv^{X_1,X_2}\|_{H^1(\R)}+\|\underline{u}-\tu^{X_1,X_2}\|_{H^1(\R)}\\
	&\le \sum_{\pm}\left(\|\underline{v}-v_\pm\|_{L^2(\R_\pm)}+\|\underline{u}-u_\pm\|_{L^2(\R_\pm)}\right)\\
	&\quad + \|\tv_1^{X_1}-v_m\|_{L^2(\R_+)} + \|\tv_1^{X_1}-v_-\|_{L^2(\R_-)} +\|\tv_2^{X_2}-v_+\|_{L^2(\R_+)}+\|\tv_2^{X_2}-v_m\|_{L^2(\R_-)}\\
	&\quad + \|\pa_x\underline{v}\|_{L^2(\R)} + \|\pa_x \tv_1^{X_1}\|_{L^2(\R)}+\|\pa_x\tv_2^{X_2}\|_{L^2(\R)}\\
	&\quad + \|\tu_1^{X_1}-u_m\|_{L^2(\R_+)} + \|\tu_1^{X_1}-u_-\|_{L^2(\R_-)} +\|\tu_2^{X_2}-u_+\|_{L^2(\R_+)}+\|\tu_2^{X_2}-u_m\|_{L^2(\R_-)}\\
	&\quad + \|\pa_x\underline{u}\|_{L^2(\R)} + \|\pa_x \tu_1^{X_1}\|_{L^2(\R)}+\|\pa_x\tu_2^{X_2}\|_{L^2(\R)}\\
	&\le C\sqrt{\delta_1}(1+\sqrt{X_1(t)})+ C\sqrt{\delta_2}(1+\sqrt{X_2(t)}),
\end{align*}
where we used
\begin{align*}
	\int_0^\infty& |\tv_1(x-\sigma_1t-X_1(t))-v_m|^2\,d x\\
	&=\int_0^\infty |\tv_1(x-\sigma_1t)-v_m|^2\, d x+\int_{-X_1(t)}^0 |\tv_1(x-\sigma_1t)-v_m|^2\,d x\\
	&\le C\delta_1(1+|X_1(t)|),
\end{align*}
and a similar estimates for $\tv_2$. Then, using Lemma \ref{lem:xex}, we have
\[
\|\underline{v}-\tv^{X_1,X_2}\|_{H^1(\R)}+\|\underline{u}-\tu^{X_1,X_2}\|_{H^1(\R)} \le C\sqrt{\delta_0}(1+\sqrt{t}).
\]
 Therefore, taking  $0<T_1<T_0$ small enough such that $C\sqrt{\delta_0}(1+\sqrt{T_1})<\frac{\e_1}{2}$, we obtain
\begin{equation}\label{C-7}
	\|\underline{v}-\tv^{X_1,X_2}\|_{L^\infty(0,T_1;H^1(\R))}+\|\underline{u}-\tu^{X_1,X_2}\|_{L^\infty(0,T_1;H^1(\R))}\le \frac{\e_1}{2}.
\end{equation}
We now combine \eqref{C-6} and \eqref{C-7} to obtain the following $H^1$-perturbation estimate
\begin{equation}\label{C-8}
	\|v-\tv^{X_1,X_2}\|_{L^\infty(0,T_1;H^1(\R))}+\|u-\tu^{X_1,X_2}\|_{L^\infty(0,T_1;H^1(\R))}<\e_1.
\end{equation}
In particular, since the shifts $X_i(t)$ are absolutely continuous, we have 
\[v-\tv^{X_1,X_2},u-\tu^{X_1,X_2}\in C([0,T_1];H^1(\R)).\]
In order to extend the solution globally-in-time, we consider the maximal time $T_*$ defined as 
\[T_*:=\sup\left\{t>0~\Bigg|~\sup_{[0,t]}\left(\|v-\tv^{X_1,X_2}\|_{H^1(\R)}+\|u-\tu^{X_1,X_2}\|_{H^1(\R)}\right)\le \e_1\right\}.\]
Suppose $T_*<+\infty$. Then, by the continuation argument,
\begin{equation}\label{C-9}
	\sup_{[0,T_*]}\left(\|v-\tv^{X_1,X_2}\|_{H^1(\R)}+\|u-\tu^{X_1,X_2}\|_{H^1(\R)}\right)= \e_1.
\end{equation}
However, by \eqref{H1-perterb-1} and \eqref{H1-perturb-initial},
\[
\| v_0 -\tv(0,\cdot)\|_{H^1(\R)}+\| u_0 -\tu(0,\cdot)\|_{H^1(\R)} \le \frac{\frac{\e_1}{2}-C_0\delta_0^{1/4}}{C_0+1},
\]
which together with Proposition \ref{prop:H1-estimate}, implies that
\begin{align*}
	\sup_{[0,T_*]}&\left(\|v-\tv^{X_1,X_2}\|_{H^1(\R)}+\|u-\tu^{X_1,X_2}\|_{H^1(\R)}\right)\\
	&\le C_0\frac{\frac{\e_1}{2}-C_0\delta_0^{1/4}}{C_0+1} +C_0\delta_0^{1/4} \\
	&\le \left(\frac{\e_1}{2}-C_0\delta_0^{1/4} \right)+C_0\delta_0^{1/4}= \frac{\e_1}{2}.
\end{align*}
This contradicts to \eqref{C-9} and therefore, we conclude that $T_*=+\infty$. Then, it again follows from Proposition \ref{prop:H1-estimate} that
\begin{equation} \label{tempo1}
\begin{aligned}
\sup_{t >0} &\left( \norm{v-\tv^{X_1,X_2}}_{H^1(\R)}+\norm{u-\tu^{X_1,X_2}}_{H^1(\R)} \right)+\sqrt{\int_0^t \sum_{i=1}^2\delta_i |X_i|^2\,ds}\\
	&\quad +\sqrt{\int_0^\infty \left(  \mathcal{G}^S(U)+D(U)+D_1(U)+D_2(U) \right)\,ds}\\
	&\le C_0\left(\|v_0-\tv(0,\cdot)\|_{H^1(\R)}+\|u_0-\tu(0,\cdot)\|_{H^1(\R)}\right)+C_0\delta_0^{1/4},
\end{aligned}
\end{equation}
and
\beq\label{conx12}
|\dot{X}_1(t)| + |\dot{X}_2(t)|\le C_0 \|(v-\tv^{X_1,X_2})(t,\cdot)\|_{L^\infty(\bbr)},\qquad t>0.
\eeq

Especially,  since the right-hand side of \eqref{tempo1} is small enough,  we find that
\begin{equation} \label{tempo2}
\frac{v_-}{3} < v(t,x) < 3v_+,  \quad \forall(t,x) \in [0,\infty) \times \R.
\end{equation}

\subsection{Time-asymptotic behavior}\label{subsec:time-asymptotic}
Based on the global estimates \eqref{tempo1}-\eqref{tempo2}, we will show the time-asymptotic behavior in Theorem \ref{thm:main}. 
The proof mainly follows the typical argument, but we will crucially use the following two lemmas based on the wave separation by shifts, since the energy functional $\mathcal{G}^S$ is localized by the cutoff functions $\phi_i$.\\
First, the estimate \eqref{tempo1} with the Sobolev embedding implies
\[
\|v-\tv^{X_1,X_2}\|_{L^\infty(\bbr_+\times\bbr)} \le C(\eps_0+\delta_0^{1/4}).
\]
Thus, \eqref{conx12} and the smallness of $\delta_0$ and $ \eps_0$ imply 
\[
X_1(t)\le -\frac{\sigma_1}{2}t,\quad X_2(t)\ge -\frac{\sigma_2}{2}t,\quad t>0,
\] 
or equivalently,
\beq\label{sepx12}
X_1(t)+\sigma_1t \le \frac{\sigma_1}{2}t,\quad X_2(t) +\sigma_2t \ge \frac{\sigma_2}{2}t,\quad t>0,
\eeq
which proves \eqref{sx12}.
Thus, the waves $\widetilde v_1^{X_1}$ and $\widetilde v_2^{X_2}$ are well-separated, and so are $\widetilde v_1^{X_1}$ and $\phi_2$. Using this property, we have the following lemmas on the wave separation.

\begin{lemma}\label{lem:shock-interact-1}
Assume \eqref{sepx12}.	Given $v_+>0$, there exist positive constants $\delta_0, C$ such that for any $\delta_1, \delta_2\in(0,\delta_0)$, the following estimates hold. For each $i=1, 2$,
	\begin{align*}
	& |(\tv_i)^{X_i}_x||\tv^{X_1,X_2}-\tv^{X_i}_i| \le C \delta_i \delta_1\delta_2\exp\left(-C\min\{\delta_1,\delta_2\}t\right),\quad t>0, \quad x\in\bbr,\\
		&\int_{\bbr}|(\tv_i)^{X_i}_x||\tv^{X_1,X_2}-\tv^{X_i}_i|\,dx\le C\delta_1\delta_2\exp\left(-C\min\{\delta_1,\delta_2\}t\right),\quad t>0,\\
		&\int_{\bbr}|(\tv_1)^{X_1}_x||(\tv_2)^{X_2}_x|\,dx\le C\delta_1\delta_2\exp\left(-C\min\{\delta_1,\delta_2\}t\right),\quad t>0.
	\end{align*}
\end{lemma}

\begin{lemma}\label{lem:shock-interact-2}
Assume \eqref{sepx12}.	Let $\phi_i$ be the functions defined in \eqref{localization}. Given $v_+>0$, there exist positive constants $\delta_0, C$ such that for any $\delta_1, \delta_2\in(0,\delta_0)$, the following estimates hold.
	\begin{align*}
		&\phi_2 |(\widetilde{v}_1)^{X_1}_x| \leq C \delta_1^2 \exp(-C \delta_1 t),\quad  \phi_1 |(\widetilde{v}_2)^{X_2}_x| \leq C \delta_2^2 \exp(-C \delta_2 t),\quad  t>0,\quad x\in\bbr,\\
		&\int_{\bbr}|(\tv_1)_x^{X_1}|\phi_2\,d x\le C\delta_1\exp(-C\delta_1t), \quad \int_{\bbr}|(\tv_2)_x^{X_2}|\phi_1\,d x\le C\delta_2\exp(-C\delta_2t),\quad  t>0.\\
	\end{align*}
\end{lemma}

The proofs of the above lemmas are straightforward from the ones of \cite[Lemma 7.3]{KV-2shock} and \cite[Lemma 7.4]{KV-2shock} respectively. However, we present the proofs of the above lemmas in the appendices for readers' convenience. \\

To get the desired result, we follow the classical method as in \cite{KVW3}. But, we here present its details for readers.  Consider a function $g$ defined by
\[g(t):=\|(v-\tv^{X_1,X_2})_x\|_{L^2(\R)}^2+\|(u-\tu^{X_1,X_2})_{x}\|_{L^2(\R)}^{2}.\]
We will prove that $g\in W^{1,1}(\bbr_+)$, which then implies 
\[\lim_{t\to\infty} g(t)=0,\]
which together with the interpolation inequlaity and  \eqref{tempo1} yields
\begin{equation}\label{c-10}
\lim_{t\to\infty}(\|v-\tv^{X_1,X_2}\|_{L^\infty(\R)}+\|u-\tu^{X_1,X_2}\|_{L^\infty(\R)})=0.
\end{equation} 
Moreover,   \eqref{conx12} and \eqref{c-10} imply that
\[
\lim_{t \rightarrow +\infty} |\dot{X}_i(t)| \leq C_0 \Vert (v-\tv^{X_1,X_2})(t) \Vert_{L^\infty(\R)}=0,  \quad i=1,2.\]
Those give the desired result.
Hence, it remains to show that $g\in W^{1,1}(\bbr_+)$.\\
	Since
	\[
	\begin{aligned}
	(p(v)-p(\tv^{X_1,X_2}))_x&=p'(v)(v-\tv^{X_1,X_2})_x+(\tv_1^{X_1}+\tv_2^{X_2})_x(p'(v)-p'(\tv^{X_1,X_2})),
	\end{aligned}
	\]
	the uniform bound \eqref{tempo2} yields
	\begin{equation} \label{tempo3}
	|(v-\tv^{X_1,X_2})_x| \leq C|(p(v)-p(\tv^{X_1,X_2}))_x|+C\left( |(\tv_1)^{X_1}_x|+|(\tv_2)^{X_2}_x| \right) |v-\tv^{X_1,X_2}|.
	\end{equation}
	Using \eqref{tempo3} with $\phi_1+\phi_2=1$ (defined as in \eqref{localization}), and then Lemma \ref{lem:shock-interact-2} and \eqref{tempo1}, we have 
	\[
\begin{aligned}	
	\int_0^\infty |g(t)| \, dt 
	&\leq C\int_0^\infty \left(\int_\R |(\tv_1)^{X_1}_x| |\phi_1 (v-\tv^{X_1,X_2})|^2 \,dx +\int_\R |(\tv_2)^{X_2}_x| |\phi_2(v-\tv^{X_1,X_2})|^2 \,dx \right.\\
	&\qquad + \int_\R |(\tv_1)^{X_1}_x| |\phi_2 (v-\tv^{X_1,X_2})|^2 \,dx +\int_\R |(\tv_2)^{X_2}_x| |\phi_1(v-\tv^{X_1,X_2})|^2 \,dx \\
	&\qquad \left.   + \int_\R |\partial_x (p(v)-p(\tv^{X_1,X_2}))|^2 \,dx +\int_\R |(u-\tu^{X_1,X_2})_x|^2 \,dx \right) \,dt\\
	&\leq C \int_0^\infty ( \mathcal{G}^S (U)+{D}(U)+D_1(U))\,dt\\
	&\qquad  +C  \norm{v-\tv^{X_1,X_2}}_{L^\infty(0,\infty;L^2(\R))} \int_0^\infty  \int_{\bbr} (|(\tv_1)^{X_1}_x| \phi_2  +    |(\tv_2)^{X_2}_x| \phi_1)\,dxdt ,\\
	&< \infty.
	\end{aligned}\]
This implies $g\in L^1(\bbr_+)$. To show $g'\in L^1(\bbr_+)$, we will use the following system satisfied by the shifted composite wave :
\beq \label{com-eqn}
\begin{cases}
\tv^{X_1,X_2}_t-\tu^{X_1,X_2}_x=-\sum_{i=1}^2 \dot{X}_i (t) (\tv_i)^{X_i}_x\\
\tu^{X_1,X_2}_t+p(\tv^{X_1,X_2})_x=\left(\frac{\tu^{X_1,X_2}_x}{\tv^{X_1,X_2}}\right)_x-\sum_{i=1}^2 \dot{X_i}(\tu_i)^{X_i}_x+E_2+E_3,
\end{cases}
\eeq
where
\[E_2:=p(\tv^{X_1,X_2})_x-\sum_{i=1}^2 p(\tv_i^{X_i})_x, \quad E_3:=-\left( \frac{\tu^{X_1,X_2}_x}{\tv^{X_1,X_2}}\right)_x+\sum_{i=1}^2 \left( \frac{(\tu_i)^{X_i}_x}{\tv_i^{X_i}}\right)_x.\]
Then it follows from \eqref{eq:NS} and \eqref{com-eqn} that
	\[
	\begin{aligned}
	&(v-\tv^{X_1,X_2})_t-\sum_{i=1}^2 \dot{X}_i (\tv_i)^{X_i}_x -(u-\tu^{X_1,X_2})_x=0,\\
	&(u-\tu^{X_1,X_2})_t-\sum_{i=1}^2 \dot{X}_i (\tu_i)^{X_i}_x+(p(v)-p(\tv^{X_1,X_2}))_x=\left(\frac{u_x}{v}-\frac{\tu^{X_1,X_2}_x}{\tv^{X_1,X_2}}\right)_x-E_2-E_3.
	\end{aligned}
	\]
	This with integration by parts yields
	\[
	\begin{aligned}
	&\int_0^\infty |g'(t)| \,dt =\int_0^\infty 2 \left| \int_\bbr (v-\tv^{X_1,X_2})_x (v-\tv^{X_1,X_2})_{xt} \,dx+\int_\bbr (u-\tu^{X_1,X_2})_x(u-\tu^{X_1,X_2})_{xt} \,dx \right| \,dt\\
	&\leq 2\int_0^\infty \int_{\bbr} \left|  (v-\tv^{X_1,X_2})_x \left[ \sum_{i=1}^2 \dot{X}_i(t) (\tv_i)^{X_i}_{xx} +(u-\tu^{X_1,X_2})_{xx} \right] + (u-\tu^{X_1,X_2})_x \left[ \sum_{i=1}^2 \dot{X}_i(t)(\tu_i)^{X_i}_{xx} \right] \right| dxdt\\
	&\qquad +2 \int_0^\infty \int_{\bbr} \left| (u-\tu^{X_1,X_2})_{xx} \left[-(p(v)-p(\tv^{X_1,X_2}))_x+\left(\frac{u_x}{v}-\frac{\tu_x^{X_1,X_2}}{\tv^{X_1,X_2}}\right)_x-E_2-E_3 \right]  \right| dxdt\\
	&\leq C \int_0^\infty \left( \int_{\bbr} |(v-\tv^{X_1,X_2})_x|^2 +|(u-\tu^{X_1,X_2})_x|^2  + |(u-\tu^{X_1,X_2})_{xx}|^2 + |\partial_x(p(v)-p(\tv^{X_1,X_2}))|^2  dx \right.\\
	&\quad \left. +\int_{\bbr} \sum_{i=1}^2 |\dot{X}_i(t)|^2  \left[ |(\tv_i)^{X_i}_{xx}|^2 + |(\tu_i)^{X_i}_{xx}|^2 \right] dx + \int_{\bbr} \left[ \left| \left(\frac{u_x}{v}-\frac{\tu_x^{X_1,X_2}}{\tv^{X_1,X_2}}\right)_x \right|^2+|E_2|^2+|E_3|^2 \right] dx\right)dt\\
	\end{aligned}
	\]
Using the same estimates as before, \eqref{tempo1} and Lemma \ref{lem:shock-est}, we first have
\[
\begin{aligned}
\int_0^\infty |g'(t)| \,dt &\leq C \int_0^\infty \left( \sum_{i=1}^2 |\dot{X}_i(t)|^2 + \mathcal{G}^S (U)+{D}(U)+D_1(U)+D_2(U)\right)\,dt + C\\
&\qquad + \int_{\bbr} \left[ \left| \left(\frac{u_x}{v}-\frac{\tu_x^{X_1,X_2}}{\tv^{X_1,X_2}}\right)_x \right|^2+|E_2|^2+|E_3|^2 \right] dx dt
\end{aligned}
\]
For the last three terms above,  we get further estimate as follows:\\
	Using  \eqref{tempo2} and \eqref{tempo3} with Lemma \ref{lem:shock-est},  we get
	\[
	\begin{aligned}
	&\int_0^\infty \int_{\bbr} \left| \left(\frac{u_x}{v}-\frac{\tu_x^{X_1,X_2}}{\tv^{X_1,X_2}}\right)_x \right|^2 \, dxdt\\
	&=\int_0^\infty \int_{\bbr} \left| \frac{1}{v}(u-\tu^{X_1,X_2})_{xx}+(\tu^{X_1,X_2})_{xx} \left( \frac{1}{v}-\frac{1}{\tv^{X_1,X_2}} \right) -\frac{(u-\tu^{X_1,X_2})_x}{v^2}(v-\tv^{X_1,X_2})_x \right.\\
	&\qquad \left.  -\frac{\tu^{X_1,X_2}_x}{v^2}(v-\tv^{X_1,X_2})_x -\frac{\tv_x^{X_1,X_2}}{v^2}(u-\tu^{X_1,X_2})_x+(\tv^{X_1,X_2})_x(\tu^{X_1,X_2})_x \left( \frac{1}{(\tv^{X_1,X_2})^2}-\frac{1}{v^2}\right) \right|^2 \,dxdt\\
	&\leq C \int_0^\infty \int_{\bbr} \bigg[ |(u-\tu^{X_1,X_2})_{xx}|^2+(|(\tv_1)^{X_1}_x|^2+|(\tv_2)^{X_2}_x|^2)|v-\tv^{X_1,X_2}|^2+| \partial_x ( p(v)-p(\tv^{X_1,X_2})) |^2 \\
	&\qquad +|(u-\tu^{X_1,X_2})_x|^2 |(v-\tv^{X_1,X_2})_x|^2+|(u-\tu^{X_1,X_2})_x|^2 \bigg] \,dxdt.
		\end{aligned}
\]
Using \eqref{tempo1} with the same estimates as before, we have
\[
\begin{aligned}
	\int_0^\infty \int_{\bbr} \left| \left(\frac{u_x}{v}-\frac{\tu_x^{X_1,X_2}}{\tv^{X_1,X_2}}\right)_x \right|^2 \, dxdt
	&\leq C+ C \int_0^\infty ( \mathcal{G}^S (U)+ {D}(U)+D_1(U)+D_2(U))\,dt \\
	&\qquad  +C \int_0^\infty \norm{(u-\tu^{X_1,X_2})_x}_{L^\infty(\R)}^2 \int_{\bbr} |(v-\tv^{X_1,X_2})_x|^2 \,dxdt.\\
		\end{aligned}
\]		
The last term is estimated by using the interpolation inequality with \eqref{tempo1} as
	\[
	\begin{aligned}
	&\int_0^\infty \norm{(u-\tu^{X_1,X_2})_x}_{L^\infty(\R)}^2 \int_{\bbr} |(v-\tv^{X_1,X_2})_x|^2 \,dxdt\\
	&\leq C \int_0^\infty \norm{(u-\tu^{X_1,X_2})_x}_{L^2 (\R)} \norm{(u-\tu^{X_1,X_2})_{xx}}_{L^2 (\R)} \norm{(v-\tv^{X_1,X_2})_x}_{L^2(\R)}^2 \,dt\\
	&\leq C \int_0^\infty \left[ \norm{(u-\tu^{X_1,X_2})_x}_{L^2 (\R)}^2+ \norm{(u-\tu^{X_1,X_2})_{xx}}_{L^2 (\R)}^2 \norm{(v-\tv^{X_1,X_2})_x}_{L^2(\R)}^4 \right] \,dt\\
	&\leq C \int_0^\infty \left[ \norm{(u-\tu^{X_1,X_2})_x}_{L^2 (\R)}^2+ \norm{(u-\tu^{X_1,X_2})_{xx}}_{L^2 (\R)}^2  \right] \,dt\\
	&\leq C \int_0^\infty \left( D_1(U)+D_2(U) \right) \,dt.
	\end{aligned}
	\]		
	Similarly,  using Lemma \ref{lem:shock-interact-1},  we have
	\[
\begin{aligned}	
&\int_0^\infty \int_{\bbr} |E_3|^2 \,dxdt =\int_0^\infty \int_{\bbr}  \left| \sum_{i=1}^2 \left(\frac{(\tu_i^{X_i})_x}{\tv_i^{X_i}}\right)_x-\left(\frac{\tu_x^{X_1,X_2}}{\tv^{X_1,X_2}}\right)_x \right|^2 \,dxdt\\
& \leq C  \int_0^\infty \int_{\bbr} \Big(   (|(\tu_1)^{X_1}_{xx}|+|(\tu_1)^{X_1}_x| |(\tv_1)^{X_1}_x|)|\tv^{X_1,X_2}-\tv_1^{X_1}|\\
&\hspace{1cm}+(|(\tu_2)^{X_2}_{xx}|+|(\tu_2)^{X_2}_x| |(\tv_2)^{X_2}_x|)|\tv^{X_1,X_2}-\tv_2^{X_2}| +   |(\tu_1)^{X_1}_x||(\tv_2)^{X_2}_x|+|(\tu_2)^{X_2}_x||(\tv_1)^{X_1}_x|\Big)^2 \, dxdt\\
& \leq C \int_0^\infty  \int_{\bbr}  \Big(   (|(\tv_1)^{X_1}_{xx}|+|(\tv_1)^{X_1}_x|^2)|\tv^{X_1,X_2}-\tv_1^{X_1}|+(|(\tv_2)^{X_2}_{xx}|+|(\tv_2)^{X_2}_x|^2 )|\tv^{X_1,X_2}-\tv_2^{X_2}| \\
&\qquad  +|(\tv_1)^{X_1}_x| |(\tv_2)^{X_2}_x| \Big)^2 dxdt\\
&\leq C \int_0^\infty \left( \Vert |(\tv_1)^{X_1}_x | \tv^{X_1,X_2}-\tv_1^{X_1}| + |(\tv_2)^{X_2}_x| |\tv^{X_1,X_2}-\tv_2^{X_2}|+|(\tv_1)^{X_1}_x| |(\tv_2)^{X_2}_x| \Vert_{L^2(\R)}^2 \right) \,dt <\infty,
\end{aligned}
\]
and
\[
\begin{aligned}
\int_0^\infty \int_{\bbr} |E_2|^2 \,dxdt&=\int_0^\infty \int_{\bbr} \left| \sum_{i=1}^2 p(\tv_i^{X_i})_x-p(\tv^{X_1,X_2})_x\right|^2 \,dxdt\\
&\leq C \int_0^\infty \int_{\bbr} \left( |p'(\tv^{X_1,X_2})-p'(\tv_1^{X_1})| |(\tv_1)^{X_1}_x|+|p'(\tv^{X_1,X_2})-p'(\tv_2^{X_2})| |(\tv_2)^{X_2}_x| \right)^2 dxdt\\
&\leq C \int_0^\infty \Vert |(\tv_1)^{X_1}_x| |\tv^{X_1,X_2}-\tv_1^{X_1}| + |(\tv_2)^{X_2}_x| |\tv^{X_1,X_2}-\tv_2^{X_2}| \Vert_{L^2(\R)}^2 \,dt <\infty.
\end{aligned}
\]
Hence, $g'\in L^1(\bbr_+)$ which completes the proof.\\

Therefore, according to Section \ref{subsec:global} and Section \ref{subsec:time-asymptotic}, it remains to prove Proposition \ref{prop:H1-estimate} to obtain global-in-time well-posedness and the time-asymptotic behavior, as in Theorem \ref{thm:main}. In the following two sections, we mainly focus on proving Proposition \ref{prop:H1-estimate}, by using the method of $a$-contraction with shifts.

\subsection{Notations} In what follows, we use the following notations for simplicity. \\
1. $C$ denotes a positive $O(1)$-constant that may change from line to line, but is independent of the small constants $\delta_0, \eps_1, \delta_1,\delta_2$, $\lambda$ (to be introduced below) and the time $T$.\\
2. We omit the dependence on the pair of shifts $(X_1,X_2)$ of the composite wave \eqref{composite_wave} without confusion as:
\[
(\tv, \tu) (t,x) := (\tv^{X_1,X_2}, \tu^{X_1,X_2}) (t,x).
\]

\section{Estimates on weighted relative entropy with shifts}\label{sec:4}
\setcounter{equation}{0}
We will use the method of $a$-contraction with shifts to obtain the $L^\infty(0,T;L^2(\bbr))$-bounds of perturbations in \eqref{C-3}. For simplicity of our analysis,  we consider the effective velocity $h:=u-(\ln v)_x$ as in \cite{KVW3} (see also  \cite{Kang-V-NS17,KV-Inven}) to transform the Navier-Stokes equations \eqref{eq:NS} to the system:
\begin{align}
\begin{aligned}\label{eq:NS-h}
&v_t - h_x = \left(\ln v\right)_{xx},\\
&h_t + p(v)_x =0.
\end{aligned}
\end{align}
Then, it follows from \eqref{viscous-shock-u} that the associated viscous shocks $\tvi$ and $\thi:=\tu_i-(\ln \tvi)_x$ satisfy the following ODEs:
\begin{equation}\label{viscous-shock-h}
\begin{cases}
-\sigma_i(\tv_i)' - (\widetilde{h}_i)' = (\ln \tv_i)'',\\
-\sigma_i(\widetilde{h}_i)' + (p(\tv_i))' = 0,\\
(\tv_1,\widetilde{h}_1)(-\infty) = (v_-,u_-),\quad (\tv_1,\widetilde{h}_1)(+\infty) = (v_m,u_m),\\
(\tv_2,\widetilde{h}_2)(-\infty) = (v_m,u_m),\quad (\tv_2,\widetilde{h}_2)(+\infty) = (v_+,u_+).
\end{cases}
\end{equation} 
Let $(\tv,\widetilde{h})$ denote the shifted composite wave as
\[
(\tv,\widetilde{h})(t,x) := \left(\tv_1^{X_1}(x-\sigma_1t)+\tv_2^{X_2}(x-\sigma_2t)-v_m,\widetilde{h}_1^{X_1}(x-\sigma_1t)+\widetilde{h}_2^{X_2}(x-\sigma_2t)-u_m\right),
\]
which satisfies
\beq \label{com-eqn-h}
\begin{cases}
	\tv_t-\widetilde{h}_x=-\sum_{i=1}^2 \dot{X}_i (t) (\tv_i)^{X_i}_x +E_1\\
	\widetilde{h}_t+p(\tv)_x=-\sum_{i=1}^2 \dot{X_i}(\widetilde{h}_i)^{X_i}_x+E_2,
\end{cases}
\eeq
where $E_1$ and $E_2$ are the interaction terms defined as
\[
E_1 := -(\ln \widetilde{v})_{xx}+\sum_{i=1}^2(\ln \widetilde{v}^{X_i}_i)_{xx},\quad E_2:=p(\widetilde{v})_x-\sum_{i=1}^2p(\widetilde{v}_i^{X_i})_x.\]

\begin{remark}
	The structure of system \eqref{eq:NS-h} would be better for our analysis than using the original system \eqref{eq:NS}, since \eqref{eq:NS-h} is linear in the $h$-variable, but nonlinear in the $v$-variable that would be controlled by the parabolic term $(\ln v)_{xx}$. However, this transformation would not be crucial in our setting for small perturbation, but just for simplicity. 
\end{remark}

The goal of this section is to obtain the following control on the relative entropy.
\begin{lemma}\label{lem:rel-ent}
	There exists a positive constant $C$ such that for all $t\in [0,T]$, 
\begin{align}
	\begin{aligned}\label{est-rel-ent}
	&\int_{\R}\left(\frac{|h-\widetilde{h}|^2}{2}+Q(v|\tv)\right)\,dx+\int_0^t \left(\sum_{i=1}^2 \delta_i|\dot{X}_i|^2 + {G}_1+\mathcal{G}^S+D\right)\,ds\\
	&\le C\int_{\R}\left(\frac{|h_0(x)-\widetilde{h}(0,x)|^2}{2}+Q(v_0(x)|\tv(0,x))\right)\,dx +C\delta_0,
	\end{aligned}
\end{align}
where 
\begin{align*}
	&{G}_1(U):=\sum_{i=1}^2\int_{\bbr}|(a_i)^{X_i}_x|\left|h-\widetilde{h}-\frac{p(v)-p(\widetilde{v})}{\sigma_i}\right|^2\,dx,\\
	&\mathcal{G}^S(U):=\sum_{i=1}^2 \int_\R |(\widetilde{v}_i)^{X_i}_x| |\phi_i(p(v)-p(\widetilde{v})) |^2 dx,\\
	&{D}(U):=\int_{\bbr} |\pa_x(p(v)-p(\widetilde{v}))|^2\,dx.
\end{align*}
\end{lemma}

\subsection{Construction of weight functions}
In order to deal with the two shock waves, we first introduce two weight functions $a_1, a_2$ associated with 1-shock and 2-shock respectively: for each $i=1,2,$ we define
\[a_i(x-\sigma_i t) = 1+\frac{\lambda(p(v_m)-p(\widetilde{v_i}(x-\sigma_it)))}{\delta_i},\]
where $\delta_i$ denotes the strength of $i$-shock. Here, $\lambda$ is a small constant chosen as $\delta_i\ll\lambda<C\sqrt{\delta_i}$ for $i=1,2$ so that it is large enough compared to the strengths of shocks.\\
Notice that $\frac{1}{2}<1-\lambda\le a_i \le 1$ and
\beq\label{avd}
(a_i)_x = -\frac{\lambda}{\delta_i}(p(\widetilde{v_i}))_x,
\eeq
from which we have
\beq\label{inta}
|(a_i)_x| \sim \frac{\lambda}{\delta_i} |(\widetilde{v_i})_x|,\quad \mbox{and so,}\quad  \|(a_i)_x\|_{L^\infty(\bbr)}\le \lambda\delta_i,\quad  \|(a_i)_x\|_{L^1(\bbr)} =\lambda.
\eeq
To handle the shifted composite wave, we consider the following composition of shifted weight functions as
\begin{equation}\label{weight}
	a^{X_1,X_2}(t,x):=a_1^{X_1}(x-\sigma_1 t) +a_2^{X_2}(x-\sigma_2t)-1=a_1(x-\sigma_1t-X_1(t))+a_2(x-\sigma_2t-X_2(t))-1.
\end{equation}
For notational simplicity as before, we will omit the dependence on shifts:
\[
a(t,x) := a^{X_1,X_2}(t,x).
\]

\subsection{Estimate on shifts}
We here show the estimate \eqref{bddx12} on the derivative of shifts. First, take $\eps_1$ and $\delta_0$ small enough such that $\eps_1, \delta_0 \in (0,\delta_*)$, where $\delta_*$ is the constant as in Lemma \ref{lem-rel-quant} so that we can apply Lemma \ref{lem-rel-quant} to the proof of the main theorem.\\
We use the assumption \eqref{perturbation_small} with the Sobolev inequality and \eqref{est-rel-2} to have
\beq\label{smp1}
\|p(v)-p(\tv)\|_{L^\infty((0,T)\times\bbr)}\le C\|v-\tv\|_{L^\infty((0,T)\times\bbr)} \le C\eps_1.
\eeq
Then, using \eqref{smp1} with $\sigma_i(\widetilde{h}_i)' =p(\tv_i)' $ from \eqref{viscous-shock-h}, it follows from \eqref{X(t)} that for each $i$,
\beq\label{dxbound}
|\dot{X}_i(t)| \le \frac{C}{\delta_i}  \||p(v)-p(\tv)|+|v-\tv| \|_{L^\infty(\bbr)} \int_\bbr (\tv_i)^{X_i}_x dx \le C\|v-\tv\|_{L^\infty(\bbr)},\quad t\le T.
\eeq
This gives the desired estimate \eqref{bddx12}. As in \eqref{sepx12}, the above estimates imply
\beq\label{fsx12}
X_1(t)+\sigma_1t \le \frac{\sigma_1}{2}t,\quad X_2(t) +\sigma_2t \ge \frac{\sigma_2}{2}t,\quad t\le T.
\eeq

\subsection{Relative entropy method}\label{sec:re}
First of all, we rewrite the system \eqref{eq:NS-h} in the abstract form:
\beq\label{system-vh}
\partial_t U +\partial_x A(U)= \partial_x \Big(M(U) \partial_x D\eta(U) \Big),
\eeq
where
\[
U:={v \choose h},\quad A(U):={-h \choose p(v)},\quad M(U):={\frac{1}{\gamma p(v)} \quad 0 \choose\quad\  0\quad\ 0},
\]
and 
\beq\label{nablae}
D\eta(U)={-p(v)\choose h}.
\eeq
Indeed, by $-p'(v)v=\gamma p(v)$, one has
\[
\big(\ln v\big)_{xx}=\left(\frac{(-p(v))_x}{\gamma p(v)}\right)_{x}.
\]
Then, using the nonnegative matrix $M(U)$ and the gradient \eqref{nablae} of the entropy $\eta(U):=\frac{h^2}{2}+Q(v)$ of \eqref{eq:NS-h}, where $Q(v)=\frac{v^{-\gamma+1}}{\gamma-1}$, i.e., $Q'(v)=-p(v)$,
we have
\[
{ \big(\ln v\big)_{xx} \choose 0} = \partial_x\Big(M(U) \partial_x D\eta(U) \Big).
\]
Let $\widetilde{U}_1(x-\sigma_1 t)$ and $\widetilde{U}_2(x-\sigma_2t)$ be the viscous 1-shock and 2-shock of the system \eqref{system-vh}. 
Then, the system \eqref{viscous-shock-h} can be cast in the following form: for each $i=1,2,$
\begin{equation}\label{viscous-shock}
	-\sigma_i(\widetilde{U}_i)_x +A(\widetilde{U}_i)_x = \pa_x\left(M(\widetilde{U}_i)\pa_x D\eta(\widetilde{U}_i)\right).
\end{equation}
Then, from  \eqref{com-eqn-h}, the shifted composite wave $\tU:=(\tv,\widetilde{h})$ satisfies
	\beq\label{sysh}
	\pa_t\tU +A(\tU)_x=\pa_x(M(\tU)\pa_xD\eta(\tU))-\sum_{i=1}^2\dot{X}_i(\widetilde{U}_i)^{X_i}_x +\begin{pmatrix}
		E_1\\E_2
	\end{pmatrix}.\eeq
	
For the relative entropy method, we use the relative entropy functional $\eta(U|V)$  and relative flux $A(U|V)$ defined as 
\[\eta(U|V):= \eta(U)-\eta(V)-D\eta(V)(U-V)\]
and
\[A(U|V):=A(U)-A(V)-DA(V)(U-V).\]
We also define the relative entropy flux $G(U;V)$ as
\[G(U;V):=G(U)-G(V)-D\eta(V)(A(U)-A(V)),\]
where $G$ is the entropy flux for $\eta$ satisfying $\pa_iG(U) = \sum_{k=1}^2 \pa_k\eta(U)+\pa_iA_k(U)$ for $i=1,2$. The relative quantities for the system \eqref{eq:NS-h} can be explicitly written as follows:
\begin{align*}
	&\eta(U|\widetilde{U}) = \frac{|h-\widetilde{h}|^2}{2}+Q(v|\widetilde{v}),\\
	&A(U|\widetilde{U}) = \begin{pmatrix}0\\ p(v|\widetilde{v})\end{pmatrix},\\
	&G(U;\widetilde{U}) = (p(v)-p(\widetilde{v}))(h-\widetilde{h}).
\end{align*}
Here, the relative internal energy $Q(v|\widetilde{v})$ and the relative pressure $p(v|\widetilde{v})$ are defined as
\[Q(v|\widetilde{v}) = Q(v)-Q(\widetilde{v}) -Q'(\widetilde{v})(v-\widetilde{v}),\quad \mbox{and}\quad p(v|\widetilde{v}) = p(v)-p(\widetilde{v}) -p'(\widetilde{v})(v-\widetilde{v}).\]

In order to obtain the desired estimate in Lemma \ref{lem:rel-ent}, we need to control the relative entropy between the solution $U$ to \eqref{system-vh} and the shifted composite wave $\widetilde{U}$. However, as we mentioned, the direct estimate of the relative entropy would fail to control it, and therefore, we instead estimate the following weighted relative entropy:
\[\int_\bbr a(t,x)\eta\left(U(t,x)\Big|\tU(t,x)\right)\,d x.\]
It follows from the choice of the weight function that the global weight $a(t,x)$ satisfies $1/2\le a\le 1$ by the smallness of $\lambda$, which recovers the control of the relative entropy. Following the estimate in \cite[Lemma 5.2]{KV-2shock} (see also \cite[Lemma 4.3]{KVW3} and \cite[Lemma 4]{Vasseur-Book}), the time-derivative of the weighted relative entropy can be computed as follows.

\begin{lemma}\label{lem:req}
	Let $a := a^{X_1,X_2}$ be the weight function defined in \eqref{weight} and $X_1$, $X_2$ be any absolutely continuous function. Let $U$ be a solution to system \eqref{system-vh} and $\tU$ be the shifted composition wave satisfying \eqref{sysh}. Then, we have
	\[
		\frac{d}{dt}\int_{\bbr}a(t,x)\eta\left(U(t,x)\Big|\widetilde{U}(t,x)\right)\,dx = \sum_{i=1}^2(\dot{X}_i(t)Y_i(U))+\mathcal{J}^{\rm bad}(U)-\mathcal{J}^{\rm good}(U),
	\]
	where
	\beq\label{yi}
	Y_i(U):=-\int_{\bbr} (a_i)^{X_i}_x\eta(U|\widetilde{U})\,dx +\int_\bbr a(\widetilde{U}_i)^{X_i}_xD^2\eta(\widetilde{U})(U-\widetilde{U})\,dx,
	\eeq
	and
	\begin{align*}
		&\mathcal{J}^{\rm bad}(U):=\sum_{i=1}^2\Bigg[\int_{\bbr}(a_i)^{X_i}_x (p(v)-p(\widetilde{v}))(h-\widetilde{h})\,dx + \sigma_i\int_\bbr a(\widetilde{v}_i)^{X_i}_xp(v|\widetilde{v})\,dx\\
		&\hspace{3cm}-\int_{\bbr}(a_i)^{X_i}_x\frac{p(v)-p(\widetilde{v})}{\gamma p(v)}\pa_x(p(v)-p(\widetilde{v}))\,dx\\
		&\hspace{3cm}-\int_{\bbr}(a_i)^{X_i}_x(p(v)-p(\widetilde{v}))^2\frac{\pa_xp(\widetilde{v})}{\gamma p(v)p(\widetilde{v})}\,dx\Bigg]\\
		&\hspace{2cm}-\int_{\bbr}a\pa_x(p(v)-p(\widetilde{v}))\frac{p(v)-p(\widetilde{v})}{\gamma p(v)p(\widetilde{v})}\pa_xp(\widetilde{v})\,dx\\
		&\hspace{2cm}+\int_{\bbr}a(p(v)-p(\widetilde{v}))E_1\,dx-\int_{\bbr}a(h-\widetilde{h})E_2\,dx,\\
		&\mathcal{J}^{\rm good}(U):=\sum_{i=1}^2\left[\frac{\sigma_i}{2}\int_{\bbr}(a_i)^{X_i}_x\left|h-\widetilde{h}\right|^2\,dx+\sigma_i\int_{\bbr}(a_i)^{X_i}_xQ(v|\widetilde{v})\,dx\right]\\
		&\hspace{3cm}+\int_{\bbr}\frac{a}{\gamma p(v)}|\pa_x(p(v)-p(\widetilde{v}))|^2\,dx.
	\end{align*}
\end{lemma}

\begin{remark}
We note that the definition of the weight $a_i$ implies
\[\sigma_i (a_i)^{X_i}_x = -\frac{\lambda\sigma_i}{\delta_i}(p(\widetilde{v_i}^{X_i}))_x = \frac{\gamma\lambda\sigma_i}{\delta_i}(\widetilde{v_i}^{X_i})^{-\gamma-1}(\widetilde{v}_i^{X_i})_x>0,\]
from which we observe that $\mathcal{J}^{\rm good}$ consists of good terms.
\end{remark}

\subsection{Maximization on $h-\widetilde{h}$}
Among the terms in $\mathcal{J}^{\rm bad}(U)$, a primary bad term is
\[\int_{\bbr}(a_i)^{X_i}_x(p(v)-p(\tv))(h-\widetilde{h})\,dx\] 
where the perturbations for $p(v)$ and $h$ are coupled. In order to exploit the parabolic term on $v$-variable and hence use the Poincar\'e-type inequality, we separate $h-\widetilde{h}$ from $p(v)-p(\widetilde{v})$ by using the quadratic structure of $h-\widetilde{h}$. Precisely, since
\begin{align*}
	(a_i)^{X_i}_x&(p(v)-p(\widetilde{v}))(h-\widetilde{h})-\frac{\sigma_i}{2}(a_i)^{X_i}_x|h-\widetilde{h}|^2\\
	&=-\frac{\sigma_i(a_i)^{X_i}_x}{2}\left|h-\widetilde{h}-\frac{p(v)-p(\widetilde{v})}{\sigma_i}\right|^2+\frac{(a_i)^{X_i}_x}{2\sigma_i}|p(v)-p(\widetilde{v})|^2,
\end{align*}
we rewrite the terms $\mathcal{J}^{\rm bad}(U)-\mathcal{J}^{\rm good}(U)$ in the above lemma as 
\[\mathcal{J}^{\text{bad}}(U) - \mathcal{J}^{\text{good}}(U) = \mathcal{B}(U)-\mathcal{G}(U),\]
where
\begin{align*}
	\mathcal{B}(U)&:=\sum_{i=1}^2\Bigg[\frac{1}{2\sigma_i}\int_{\bbr}(a_i)^{X_i}_x |p(v)-p(\widetilde{v})|^2\,dx + \sigma_i\int_\bbr a(\widetilde{v}_i)^{X_i}_xp(v|\widetilde{v})\,dx\\
	&\quad\qquad-\int_{\bbr}(a_i)^{X_i}_x\frac{p(v)-p(\widetilde{v})}{\gamma p(v)}\pa_x(p(v)-p(\widetilde{v}))\,dx\\
	&\quad\qquad-\int_{\bbr}(a_i)^{X_i}_x(p(v)-p(\widetilde{v}))^2\frac{\pa_xp(\widetilde{v})}{\gamma p(v)p(\widetilde{v})}\,dx\Bigg]\\
	&\quad-\int_{\bbr}a\pa_x(p(v)-p(\widetilde{v}))\frac{p(v)-p(\widetilde{v})}{\gamma p(v)p(\widetilde{v})}\pa_xp(\widetilde{v})\,dx\\
	&\quad+\int_{\bbr}a(p(v)-p(\widetilde{v}))E_1\,dx-\int_{\bbr}a(h-\widetilde{h})E_2\,dx,\\
	\mathcal{G}(U)&:=\sum_{i=1}^2\left[\frac{\sigma_i}{2}\int_{\bbr}(a_i)^{X_i}_x\left|h-\widetilde{h}-\frac{p(v)-p(\widetilde{v})}{\sigma_i}\right|^2\,dx+\sigma_i\int_{\bbr}(a_i)^{X_i}_xQ(v|\widetilde{v})\,dx\right]\\
	&\quad +\int_{\bbr}\frac{a}{\gamma p(v)}|\pa_x(p(v)-p(\widetilde{v}))|^2\,dx.
\end{align*}

Therefore, we will estimate the right-hand side of the following equation: 
\begin{equation}\label{est-1}
	\frac{d}{dt}\int_{\bbr}a\eta(U|\widetilde{U})\,dx = \sum_{i=1}^2(\dot{X}_iY_i(U)) + \mathcal{B}(U)-\mathcal{G}(U).
\end{equation}

\subsection{Decompositions}
We first  name each term of $\mathcal{B}(U)$ and $\mathcal{G}(U)$ as follows:
\begin{align*}
	&\mathcal{B}(U) =\sum_{i=1}^5 \mathcal{B}_i(U) + \mathcal{S}_1 +\mathcal{S}_2,\\
	&\mathcal{G}(U) = \mathcal{G}_1(U) + \mathcal{G}_2(U) + \mathcal{D}(U),
\end{align*}
where
\begin{align*}
	&\mathcal{B}_1(U) := \sum_{i=1}^2\frac{1}{2\sigma_i}\int_{\bbr}(a_i)^{X_i}_x |p(v)-p(\widetilde{v})|^2\,dx,\\
	&\mathcal{B}_2(U) := \sum_{i=1}^2 \sigma_i\int_\bbr a(\widetilde{v}_i)^{X_i}_xp(v|\widetilde{v})\,dx,\\
	&\mathcal{B}_3(U) := -\sum_{i=1}^2\int_{\bbr}(a_i)^{X_i}_x\frac{p(v)-p(\widetilde{v})}{\gamma p(v)}\pa_x(p(v)-p(\widetilde{v}))\,dx,\\
	&\mathcal{B}_4(U) :=-\sum_{i=1}^2\int_{\bbr}(a_i)^{X_i}_x(p(v)-p(\widetilde{v}))^2\frac{\pa_xp(\widetilde{v})}{\gamma p(v)p(\widetilde{v})}\,dx,\\
	&\mathcal{B}_5(U):=-\int_{\bbr}a\pa_x(p(v)-p(\widetilde{v}))\frac{p(v)-p(\widetilde{v})}{\gamma p(v)p(\widetilde{v})}\pa_xp(\widetilde{v})\,dx,\\
	&\mathcal{S}_1:= \int_{\bbr}a(p(v)-p(\widetilde{v}))E_1\,dx,\quad\mathcal{S}_2:=-\int_{\bbr}a(h-\widetilde{h})E_2\,dx,
\end{align*}
and 
\begin{align*}
	&\mathcal{G}_1(U):=\sum_{i=1}^2\frac{\sigma_i}{2}\int_{\bbr}(a_i)^{X_i}_x\left|h-\widetilde{h}-\frac{p(v)-p(\widetilde{v})}{\sigma_i}\right|^2\,dx,\\
	&\mathcal{G}_2(U):= \sum_{i=1}^2\sigma_i\int_{\bbr}(a_i)^{X_i}_xQ(v|\widetilde{v})\,dx,\\
	&\mathcal{D}(U):=\int_{\bbr}\frac{a}{\gamma p(v)}|\pa_x(p(v)-p(\widetilde{v}))|^2\,dx.
\end{align*}
For each $Y_i$ in \eqref{yi}, we first write $Y_i$ more explicitly as follows:
\begin{align*}
	Y_i(U)
	& = -\int_{\bbr}(a_i)^{X_i}_x\left(\frac{|h-\widetilde{h}|^2}{2}+Q(v|\widetilde{v})\right)\,dx +\int_{\bbr}a (\widetilde{h}_i)^{X_i}_x(h-\widetilde{h})\,dx\\
	&\quad -\int_{\bbr}ap'(\widetilde{v})(\widetilde{v}_i)^{X_i}_x(v-\widetilde{v})\,dx.
\end{align*}
Since a key idea in our analysis is to apply Poincar\'e-type inequality of Lemma \ref{lem-poin} by extracting a good term on an average of perturbation $p(v)-p(\tv)$ from the shift part $\sum_{i=1}^2(\dot{X}_iY_i(U))$, we decompose $Y_i$ as follows: for each $i$,
\[Y_i = \sum_{j=1}^6 Y_{ij},\]
where
\begin{align*}
	&Y_{i1} := \int_{\bbr}\frac{a}{\sigma_i}(\widetilde{h}_i)^{X_i}_x(p(v)-p(\widetilde{v}))\,dx,\\
	&Y_{i2} := -\int_{\bbr}ap'(\widetilde{v}^{X_i}_i)(\widetilde{v}_i)^{X_i}_x(v-\widetilde{v})\,dx = -\int_{\bbr}a (p(\widetilde{v}^{X_i}_i))_x(v-\widetilde{v})\,dx,\\
	&Y_{i3} := \int_{\bbr}a(\widetilde{h}_i)^{X_i}_x\left(h-\widetilde{h}-\frac{p(v)-p(\widetilde{v})}{\sigma_i}\right)\,dx,\\
	&Y_{i4} := -\int_{\bbr}a(p'(\widetilde{v})-p'(\widetilde{v}^{X_i}_i))(\widetilde{v}_i)^{X_i}_x(v-\widetilde{v})\,dx,\\
	&Y_{i5} := -\frac{1}{2}\int_{\bbr}(a_i)^{X_i}_x\left(h-\widetilde{h}-\frac{p(v)-p(\widetilde{v})}{\sigma_i}\right)\left(h-\widetilde{h}+\frac{p(v)-p(\widetilde{v})}{\sigma_i}\right)\,dx,\\
	&Y_{i6} := -\int_{\bbr}(a_i)^{X_i}_xQ(v|\widetilde{v})\,dx -\int_{\bbr}\frac{(a_i)^{X_i}_x}{2\sigma_i^2}(p(v)-p(\widetilde{v}))^2\,dx.
\end{align*}
Observe that it follows from our construction \eqref{X(t)} on shifts $X_i$ that
\[\dot{X}_i = -\frac{M}{\delta_i}(Y_{i1}+Y_{i2}),\]
which yields
\begin{equation}\label{XiYi}
	\dot{X}_iY_i(U) = -\frac{\delta_i}{M}|\dot{X_i}|^2+\dot{X}_i\sum_{j=3}^6Y_{ij}.
\end{equation}
Here, the good term $-\frac{\delta_i}{M}|\dot{X_i}|^2$ would give an average of linear perturbation on $v$-variable as mentioned above, while the remaining part  would be controlled by the good terms in $\mathcal{G}(U)$. To show it, we combine \eqref{est-1} and \eqref{XiYi} to have
\begin{align}
	\begin{aligned}\label{est}
	&\frac{d}{dt} \int_\R a \eta (U| \tU) \,dx =\mathcal{R}, \quad\mbox{where}\\
	&\mathcal{R} :=-\sum_{i=1}^2\frac{\delta_i}{M}|\dot{X}_i|^2 + \sum_{i=1}^2\left(\dot{X_i}\sum_{j=3}^6Y_{ij}\right) +\sum_{i=1}^5\mathcal{B}_i +\mathcal{S}_1+\mathcal{S}_2-\mathcal{G}_1-\mathcal{G}_2-\mathcal{D}. \\
	&\quad= \underbrace{-\sum_{i=1}^2\frac{\delta_i}{2M}|\dot{X}_i|^2 + \mathcal{B}_1 +\mathcal{B}_2-\mathcal{G}_2-\frac{3}{4}\mathcal{D}}_{=:\mathcal{R}_1}\\
	&\quad\quad \underbrace{-\sum_{i=1}^2\frac{\delta_i}{2M}|\dot{X}_i|^2 + \sum_{i=1}^2\left(\dot{X_i}\sum_{j=3}^6Y_{ij}\right) +\sum_{i=3}^5\mathcal{B}_i +\mathcal{S}_1+\mathcal{S}_2-\mathcal{G}_1-\frac{1}{4}\mathcal{D}}_{=:\mathcal{R}_2}.
	\end{aligned}
\end{align}

A motivation of the decomposition \eqref{est} is that the two bad terms $\mathcal{B}_1$ and $\mathcal{B}_2$ are the main bad terms, which should be controlled by using the sharp Poincar\'e-type inequality of Lemma \ref{lem-poin}. The remaining terms in $\mathcal{R}_2$ can be controlled in a rather rough way. 
We first focus on the estimate of $\mathcal{R}_1$.

\subsection{Estimate of the main part $\mathcal{R}_{1}$}\label{sec:4.5}
A key idea for estimates of  $\mathcal{R}_{1}$ is to use the Poincar\'e-type inequality \eqref{poincare} in Lemma \ref{lem-poin}. To apply the Poincar\'e-type inequality, we need to localize the perturbation near each wave by using the cutoff functions $\phi_1, \phi_2$ defined in \eqref{localization} (see Remark \ref{rem-phi}), and then change of variables from whole space $\bbr$ to a bounded interval $(0,1)$ for each wave.  
For that, we will consider the following change of variables in space: 
\[
y_1:=1-\frac{p(v_m)-p(\widetilde{v}_1(x-\sigma_1t))}{\delta_1},\qquad y_2 := \frac{p(v_m)-p(\widetilde{v}_2(x-\sigma_2t))}{\delta_2}.
\]
Indeed, for each $i$, $y_i :\bbr\to (0,1)$ is monotone in $\xi_i:=x-\sigma_i t$, since
\[\frac{d y_1}{d\xi_1} = \frac{1}{\delta_1}p(\widetilde{v}_1)' >0,  \quad \frac{d y_2}{d\xi_2} = -\frac{1}{\delta_2}p(\widetilde{v}_2)' >0, \]
and
\[\lim_{\xi_i\to-\infty}y_i = 0,\quad \lim_{\xi_i\to+\infty} y_i = 1.\]
In terms of the new variables, we will apply the Poincar\'e-type inequality to each perturbation $w_i$ localized by $\phi_i$ respectively:
 \[
 \begin{aligned}
&w_1:=\phi_1(x+X_1(t)) \Bigg(p(v(t,x+X_1(t))) - p\Big(\tv_1(x-\sigma_1 t)+\tv_2(x-\sigma_2 t-X_2(t)+X_1(t))-v_m \Big) \Bigg) \circ y_1^{-1}, \\
&w_2:=\phi_2(x+X_2(t)) \Bigg(p(v(t,x+X_2(t))) - p\Big(\tv_1(x-\sigma_1 t-X_1(t)+X_2(t))+\tv_2(x-\sigma_2 t)-v_m \Big) \Bigg) \circ y_2^{-1}.
\end{aligned}
\]
In what follows, for simplicity, we use the following notations to denote constants of $O(1)$-scale:
\[\sigma_m:=\sqrt{-p'(v_m)},\quad \alpha_m:=\frac{\gamma+1}{2\gamma\sigma_mp(v_m)},\]
which are indeed independent of the small constants $\delta_i$ since $v_+/2\le v_m \le v_+$.\\
Since the shock strengths $\delta_i$ are bounded by $\delta_0$ respectively, the following estimates on the $O(1)$-constants hold:
\begin{equation}\label{shock_speed_est}
	|\sigma_1- (-\sigma_m)| \leq C \delta_1,\quad |\sigma_2-\sigma_m|\le C\delta_2,
\end{equation}
and
\begin{equation}\label{shock_speed_est-2}
	\|\sigma_m^2-|p'(\tv_i)|\|_{L^\infty} \le C\delta_i,\quad \left\|\frac{1}{\sigma_m^2}-\frac{p(\tv_i)^{-\frac{1}{\gamma}-1}}{\gamma}\right\|_{L^\infty}\le C\delta_i,\quad \left\|\frac{1}{\sigma_m^2}-\frac{p(\tv)^{-\frac{1}{\gamma}-1}}{\gamma}\right\|_{L^\infty}\le C\delta_0.
\end{equation}
We are now ready to estimate the terms in $\mathcal{R}_1$. 
As mentioned, we need to extract a good term on an average of the perturbation $w_i$ from the shift part $\frac{\delta_i}{2M}|\dot{X}_i|^2$  as follows, so that we could apply Lemma \ref{lem-poin}.\\

\noindent $\bullet$  {\bf(Estimate of shift part $\frac{\delta_i}{2M}|\dot{X}_i|^2$):}
We will show that: for each $i$,
		\begin{align}
		\begin{aligned} \label{L2.2}
		-\frac{\delta_i}{2M}|\dot{X}_i|^2&\le -\frac{M\delta_i}{\sigma_m^4}\left(\int_0^1 w_i\,dy_i\right)^2+C\delta_i(\lambda+\delta_0+\e_1)^2\int_0^1|w_i|^2\,dy_i\\
		&\quad +C \delta_i \exp(-C \delta_i t) \int_\R \eta(U|\widetilde{U}) \, dx.
	\end{aligned}
	\end{align}
	As the estimates of $\frac{\delta_1}{2M}|\dot{X}_1|^2$ and $\frac{\delta_2}{2M}|\dot{X}_2|^2$ are the same, we only handle the case of $X_1$. Since $\dot{X}_1=-\frac{M}{\delta_1}(Y_{11}+Y_{12})$, we first estimate $Y_{11}$ and $Y_{12}$.\\
 Using the relation $\phi_1+\phi_2=1$, we have
	\begin{align*}
		Y_{11} 
		&=\int_{\bbr}\frac{a}{\sigma_1}(\widetilde{h}_1)^{X_1}_x\phi_1(p(v)-p(\widetilde{v}))\,dx+\int_{\bbr}\frac{a}{\sigma_1}(\widetilde{h}_1)^{X_1}_x\phi_2(p(v)-p(\widetilde{v}))\,dx\\
		&=\frac{1}{\sigma_1^2}\int_{\bbr}a p(\widetilde{v}_1)^{X_1}_x \phi_1 (p(v)-p(\tv))\,dx+\int_{\bbr}\frac{a}{\sigma_1^2}p(\tv_1)^{X_1}_x\phi_2(t,x)(p(v)-p(\widetilde{v}))\,dx.
	\end{align*}
The first term is a main part for applying the Poincar\'e inequality.  We use change of variable $x\to x+X_1(t)$ and apply the change of variables for $y_1, w_1$ to observe that
\[
\frac{1}{\sigma_1^2}\int_{\bbr}a p(\widetilde{v}_1)^{X_1}_x \phi_1 (p(v)-p(\tv))\,dx = \frac{\delta_1}{\sigma_1^2}\int_0^1a(t,x+X_1(t)) w_1\,dy_1.
\]
Then, using \eqref{shock_speed_est} and $\|a-1\|_{L^\infty(\bbr_+\times\bbr)}\le \lambda$, we have
	\[\left|Y_{11}- \frac{\delta_1}{\sigma_m^2}\int_0^1w_1\,dy_1\right|\le C\delta_1(\lambda+\delta_0)\int_0^1|w_1|\,dy_1+C\int_\R | (\widetilde{v}_1)^{X_1}_x|  \phi_2 | p(v)-p(\widetilde{v}) | \,dx .\]
	To estimate $Y_{12}$, we first use the equation of state $v=p(v)^{-1/\gamma}$ and Taylor expansion to have
\[\left|v-\tv -\left(-\frac{p(\tv)^{-\frac{1}{\gamma}-1}}{\gamma}(p(v)-p(\tv))\right)\right|\le C|p(v)-p(\tv)|^2,\]
	which together with the estimates \eqref{shock_speed_est-2} and \eqref{smp1} yields
	\[\left|v-\tv -\left(-\frac{1}{\sigma_m^2}(p(v)-p(\tv))\right)\right|\le C(\delta_0+\e_1)|p(v)-p(\tv)|.\]	
This with the same argument as above implies
	\[\left|Y_{12}-\frac{\delta_1}{\sigma_m^2}\int_0^1w_1\,dy_1\right|\leq C\delta_1(\lambda+\delta_0+\e_1)\int_0^1|w_1|\,dy_1+C \int_\R |(\tv_1)^{X_1}_x| \phi_2 |p(v)-p(\tv)| \,dx .\]
	We combine the estimates for $Y_{11}$ and $Y_{12}$ to obtain
	\begin{align*}
		\left|\dot{X}_1 + \frac{2M}{\sigma_m^2}\int_0^1w_1\,dy_1\right|&\le \frac{M}{\delta_1}\left(\left|Y_{11}-\frac{\delta_1}{\sigma_m^2}\int_0^1w_1\,dy_1\right|+\left|Y_{12}-\frac{\delta_1}{\sigma_m^2}\int_0^1w_1\,dy_1\right|\right)\\
		&\le C(\lambda+\delta_0+\e_1)\int_0^1|w_1|\,dy_1+\frac{C}{\delta_1}\int_\R |(\widetilde{v}_1)^{X_1}_x | \phi_2 | p(v)-p(\tv) | dx ,
	\end{align*}
	which implies
	\[\left(\left|\frac{2M}{\sigma_m^2}\int_0^1w_1\,dy_1\right|-|\dot{X}_1|\right)^2\le C(\lambda+\delta_0+\e_1)^2\int_0^1|w_1|^2\,dy_1+\frac{C}{\delta_1^2} \left(\int_\R |(\widetilde{v}_1)^{X_1}_x | \phi_2 | (p(v)-p(\widetilde{v})) | dx \right)^2.\]
	We use an elementary inequality $\frac{x^2}{2}-y^2\le (x-y)^2$ for $x,y\ge0$ to obtain
	\begin{align}
		\begin{aligned}\label{est-X1}
		\frac{2M^2}{\sigma_m^4}\left(\int_0^1w_1\,dy_1\right)^2-|\dot{X}_1|^2&\le C(\lambda+\delta_0+\e_1)^2\int_0^1|w_1|^2\,dy_1\\
		&\quad +\frac{C}{\delta_1^2} \left(\int_\R |(\widetilde{v}_1)^{X_1}_x | \phi_2 | (p(v)-p(\widetilde{v})) | dx \right)^2.
		\end{aligned}
	\end{align}
The last term can be estimated as follows: by using Lemmas \ref{lem:shock-est}, \ref{lem-rel-quant}, \ref{lem:shock-interact-2}, 	
\begin{align*}
		\frac{C}{\delta_1^2} \left(\int_\R |(\widetilde{v}_1)^{X_1}_x | \phi_2  | (p(v)-p(\widetilde{v})) |  \,dx \right)^2 
		&\leq \frac{C}{\delta_1^2}\int_\R (|(\widetilde{v}_1)^{X_1}_x | \phi_2)^2 dx \int_\R | (p(v)-p(\widetilde{v})) |^2 \, dx \\
		&\leq \frac{C}{\delta_1^2}\sup_{t,x}\big(\phi_2^2 |(\widetilde{v}_1)^{X_1}_x |\big)  \|(\widetilde{v}_1)^{X_1}_x\|_{L^1(\bbr)}  \int_\R Q(v|\tv) \, dx \\
		&\leq C \delta_1 \exp(-C \delta_1 t) \int_\R \eta(U|\widetilde{U}) \, dx.
	\end{align*}	
	Therefore, we substitute the above estimate to \eqref{est-X1} to derive the following estimate on $\dot{X}_1$:
	\begin{align*}
		-\frac{\delta_1}{2M}|\dot{X}_1|^2 &\le -\frac{M\delta_1}{\sigma_m^4}\left(\int_0^1 w_1\,dy_1\right)^2+C\delta_1(\lambda+\delta_0+\e_1)^2\int_0^1|w_1|^2\,dy_1  +C \delta_1^2 \exp(-C \delta_1 t) \int_\R \eta(U|\widetilde{U}) \, dx,
	\end{align*}
	which is the desired estimate \eqref{L2.2}.\\
	
\noindent $\bullet$ {\bf (Estimate of the bad term $\mathcal{B}_1$ and good term $\mathcal{G}_2$):}
Recall
\begin{align*}
	&\mathcal{B}_1(U) := \sum_{i=1}^2\underbrace{\frac{1}{2\sigma_i}\int_{\bbr}(a_i)^{X_i}_x |p(v)-p(\widetilde{v})|^2\,dx}_{=:\mathcal{B}_{i1}},\\
	&\mathcal{G}_2(U):= \sum_{i=1}^2\underbrace{ \sigma_i\int_{\bbr}(a_i)^{X_i}_xQ(v|\widetilde{v})\,dx}_{=:\mathcal{G}_{i2}}.
\end{align*}
Since the estimates for the two cases $i=1,2$ are the same, we only handle the case of $i=2$ for simplicity.\\
First, we use the estimate on $Q(v|\tv)$ in Lemma \ref{lem-rel-quant} to obtain
	\begin{align*}
		\mathcal{G}_{22}&\ge \sigma_2\int_{\bbr} (a_2)^{X_2}_x \frac{p(\widetilde{v}^{X_2}_2)^{-\frac{1}{\gamma}-1}}{2\gamma}|p(v)-p(\widetilde{v})|^2\,dx\\
		&\quad -\sigma_2\int_{\bbr}(a_2)^{X_2}_x\frac{1+\gamma}{3\gamma^2}p(\widetilde{v})^{-\frac{1}{\gamma}-2}(p(v)-p(\widetilde{v}))^3\,dx\\
		& \quad + \frac{\sigma_2}{2\gamma}\int_{\bbr}(a_2)^{X_2}_x \left(p(\widetilde{v})^{-\frac{1}{\gamma}-1}-p(\widetilde{v}^{X_2}_2)^{-\frac{1}{\gamma}-1}\right)|p(v)-p(\widetilde{v})|^2\,dx.
	\end{align*}
For simplicity, let $\widehat{\mathcal{G}}_{22}$ denote the good term given as
	\[\widehat{\mathcal{G}}_{22}:=\sigma_2 \int_\R (a_2)^{X_2}_x \frac{p(\widetilde{v}^{X_2}_2)^{-\frac{1}{\gamma}-1}}{2\gamma}|p(v)-p(\widetilde{v})|^2\,dx.
	\]
Using \eqref{shock_speed_est} and \eqref{shock_speed_est-2}, we have
\begin{align*}
\mathcal{B}_{21} \leq \frac{1}{2\sigma_m} \int_\R (a_2)^{X_2}_x |\pv-p(\tv)|^2 \, dx+ \frac{C \delta_2}{2\sigma_m} \int_\R  (a_2)^{X_2}_x |p(v)-p(\tv)|^2 \, dx
	\end{align*}
	and
	\begin{align*}
		\widehat{\mathcal{G}}_{22} \geq \frac{1}{2\sigma_m}(1-C \delta_2) \int_\R (a_2)^{X_2}_x |\pv-\tpv|^2 \, dx.
	\end{align*}
Then, using $\phi_1+\phi_2=1$ and \eqref{avd}, we estimate
	\begin{align}
	\begin{aligned}\label{B21G22}
		\mathcal{B}_{21} - \widehat{\mathcal{G}}_{22}
		&\leq C\delta_2 \int_\R (a_2)^{X_2}_x |\pv-p(\tv)|^2 \, dx\\
		&\le C\delta_2\lambda \int_{\bbr} \frac{|p(\tv_2^{X_2})_x|}{\delta_2} |\phi_2(p(v)-p(\widetilde{v}))|^2\,dx+  C\lambda \int_{\bbr} |(\tv_2)^{X_2}_x| \phi_1^2|p(v)-p(\widetilde{v})|^2\,dx.
	\end{aligned}
	\end{align}
By change of variable $x\to x+X_2(t)$, the first term of the right-hand side of \eqref{B21G22} is rewritten in the new variables $w_2, y_2$:
\[
C\delta_2\lambda \int_{\bbr} \frac{|p(\tv_2^{X_2})_x|}{\delta_2} |\phi_2(p(v)-p(\widetilde{v}))|^2\,dx = C\delta_2\lambda\int_0^1 |w_2|^2\,dy_2.
\]	
Using Lemma \ref{lem:shock-interact-2}, the last term of \eqref{B21G22} can be estimated as
\begin{align*}
C\lambda \int_{\bbr} |(\tv_2)^{X_2}_x| \phi_1^2|p(v)-p(\widetilde{v})|^2\,dx \leq C \lambda  \delta^2_2 \exp(-C \delta_2 t) \int_\R \eta(U|\widetilde{U}) \, dx.
\end{align*}
Hence we have
	\[\mathcal{B}_1-\widehat{\mathcal{G}}_{2}\le \sum_{i=1}^2 C\lambda\delta_i\int_0^1|w_i|^2\,dy_i+C\lambda\delta_i^2\exp(-C\delta_it)\int_{\R}\eta(U|\widetilde{U})\,dx.\]

\noindent $\bullet$ {\bf (Estimate of the bad term $\mathcal{B}_2$):}
Recall
\begin{align*}
	&\mathcal{B}_2(U) := \sum_{i=1}^2 \underbrace{\sigma_i\int_\bbr a(\widetilde{v}^{X_i}_i)_xp(v|\widetilde{v})\,dx}_{=:\mathcal{B}_{i2}}.
\end{align*}
Again, we only handle the case of $i=2$ for simplicity.
First, we have
\begin{align*}
	\mathcal{B}_{22}= \sigma_2\int_{\R}a(\tv_2)^{X_2}_x\phi_2^2p(v|\tv)\,dx+\sigma_2\int_{\R}a(\tv_2)^{X_2}_x(1-\phi_2^2)p(v|\tv)\,dx.
\end{align*}
For the first term, using Lemma \ref{lem-rel-quant}, \eqref{shock_speed_est}, \eqref{shock_speed_est-2} and then the change of variables, we obtain
\begin{align*}
	&\sigma_2\int_{\R}a(\tv_2)^{X_2}_x\phi_2^2p(v|\tv)\,dx =\sigma_2\int_{\R}a \frac{p(\tv_2^{X_2})_x}{p'(\tv_2^{X_2})}\phi_2^2p(v|\tv) \,dx\\
	&\le \sigma_2\int_{\R}a \frac{|p(\tv_2^{X_2})_x|}{|p'(\tv_2^{X_2})|}\phi_2^2 \left(\frac{\gamma+1}{2\gamma p(\tv)}+C\e_1\right) | p(v)-p(\tv)|^2 \,dx\\
	&\le \frac{\gamma+1}{2\gamma\sigma_mp(v_m)}(1+C(\delta_0+\lambda+\e_1)) \int_\bbr |p(\tv_2^{X_2})_x|  | \phi_2(p(v)-p(\tv))|^2 \,dx \\
		&\le\delta_2\alpha_m(1+C(\delta_0+\lambda+\e_1))\int_{0}^1|w_2|^2\,dy_2,
\end{align*}
where we used the simple notation $\alpha_m=\frac{\gamma+1}{2\gamma\sigma_mp(v_m)}$.\\
Again, the last interaction term of $\cB_{22}$ can be estimated as
\begin{align*}
	\sigma_2 \int_{\bbr}a(\widetilde{v}^{X_2}_2)_x(1+\phi_2)\phi_1p(v|\widetilde{v})\,dx &\leq C \int_\R |(\widetilde{v}_2)^{X_2}_x| \phi_1 |\pv-p(\tv)|^2 \, dx \\
	&\leq C \delta_2^2 \exp(-C \delta_2 t) \int_\R \eta(U|\widetilde{U}) \, dx .
\end{align*}
Therefore, we have
\begin{align*}
\mathcal{B}_{22}\le \delta_2\alpha_m(1+C(\delta_0+\lambda+\e_1))\int_{0}^1|w_2|^2\,dy_2+ C \delta_2^2 \exp(-C \delta_2 t) \int_\R \eta(U|\widetilde{U}) \, dx,
\end{align*}
which yields
\[\mathcal{B}_2\le \sum_{i=1}^2\delta_i\alpha_m(1+C(\delta_0+\lambda+\e_1))\int_0^1|w_i|^2\,dy_i+C\delta_i^2 \exp(-C\delta_it)\int_{\R}\eta(U|\widetilde{U})\,dx.\]

Hence, combining the above estimates on $\cB_1,\cG_2$ and $\cB_2$, we have
\begin{align}
\begin{aligned} \label{L2.3}
\mathcal{B}_1-\mathcal{G}_{2} + \mathcal{B}_2	 &\le
\sum_{i=1}^2\delta_i\alpha_m(1+C(\delta_0+\lambda+\e_1))\int_0^1|w_i|^2\,dy_i+C\delta_i^2 \exp(-C\delta_it)\int_{\R}\eta(U|\widetilde{U})\,dx \\
&\quad + \sum_{i=1}^2\Bigg[-\sigma_i\int_{\R}(a_i)^{X_i}_x\frac{1+\gamma}{3\gamma^2}p(\tv)^{-\frac{1}{\gamma}-2}(p(v)-p(\tv))^3\,dx \\
		&\hspace{2cm} +\frac{\sigma_i}{2\gamma}\int_{\R}(a_i)^{X_i}_x\left(p(\tv)^{-\frac{1}{\gamma}-1}-p(\tv_i^{X_i})^{-\frac{1}{\gamma}-1}\right)|p(v)-p(\tv)|^2\,dx\Bigg].
		\end{aligned}
	\end{align}

\noindent $\bullet$ {\bf (Estimate of the diffusion term $\mathcal{D}(U)$):}
First of all, using the fact that $\phi_{1}+\phi_{2}=1$ and $1\ge \phi_{i}\ge \phi_{i}^2\ge0$ for any $i$, we separate $\mathcal{D}(U) $ into
\[
\mathcal{D}(U)  = \int_\bbr \frac{a}{\gamma p(v)} (\phi_{1}+\phi_{2}) |\partial_x \big(p(v)-p(\tv)\big)|^2 dx \ge \sum_{i=1}^2 \int_\bbr \frac{a}{\gamma p(v)}  \phi_{i}^2 |\partial_x \big(p(v)-p(\tv)\big)|^2 dx .
\]
Since Young's inequality yields: for any $\delta_*>0$ small enough,
\begin{align*}
\int_\bbr \frac{a}{\gamma p(v)} |\partial_x \big(\phi_{i}(p(v)-p(\tv))\big)|^2 dx &\le (1+\delta_*) \int_\bbr \frac{a}{\gamma p(v)}  \phi_{i}^2 |\partial_x \big(p(v)-p(\tv)\big)|^2 dx \\
&\quad + \frac{C}{\delta_*}  \int_\bbr\frac{a}{\gamma p(v)}  |\partial_x \phi_{i}|^2 |p(v)-p(\tv)|^2 dx,
\end{align*}
we have
\begin{align*}
	-\mathcal{D}(U) & \le - \frac{1}{1+\delta_*} \sum_{i=1}^2 \int_\bbr \frac{a}{\gamma p(v)} |\partial_x \big(\phi_{i}(p(v)-p(\tv))\big)|^2 dx + \frac{C}{\delta_*}  \int_\bbr\frac{a}{\gamma p(v)}  |\partial_x \phi_{i}|^2 |p(v)-p(\tv)|^2 dx\\
	& =: J_1+J_2.
\end{align*}
To write $J_1$ in terms of the variables $y_i, w_i$, we use the following estimates in the proof of \cite[Lemma 4.5]{KVW3}:
\[
\left|\frac{1}{y_i(1-y_i)}\frac{1}{\gamma p(\tv_i)}\frac{dy_i}{d x}-\frac{\delta_ip''(v_m)}{2|p'(v_m)|^2\sigma_m}\right|\le C\delta_i^2.
\]
This with $\left\|a\frac{p(\tv_i)}{p(v)}-1\right\|_{L^\infty}\le C(\delta_0+\e_1+\lambda)$ yields
\begin{align*}
	J_1&=  - \frac{1}{1+\delta_*} \sum_{i=1}^2 \int_0^1 a \frac{p(\tv_i)}{p(v)} \frac{1}{\gamma p(\tv_i)}|\partial_{y_i}w_i|^2 \left(\frac{dy_i}{dx}\right)\,dy_i\\
	&\le -  \sum_{i=1}^2(1-C(\delta_0+\e_1+\lambda+\delta_*))\left(\frac{\delta_ip''(v_m)}{2|p'(v_m)|^2\sigma_m}-C\delta_i^2\right)\int_0^1 y_i(1-y_i)|\pa_{y_i}w_i|^2\,d y_i\\
	&\le -\sum_{i=1}^2\delta_i\alpha_m (1-C(\delta_0+\e_1+\lambda+\delta_*))\int_0^1 y_i(1-y_i)|\pa_{y_i}w_i|^2\,dy_i,
\end{align*}
where we used $\frac{p''(v_m)}{2|p'(v_m)|^2\sigma_m} = \frac{\gamma+1}{2\gamma\sigma_mp(v_m)}=\alpha_m$. \\
To estimate the term $J_2$, we use the following estimate: for each $i=1, 2$,
\beq\label{phidecay}
|\partial_x \phi_{i}(t,x)| \le \frac{4}{\sigma_2- \sigma_1} \frac{1}{t},\quad\forall x\in\bbr, \quad t\in (0, T].
\eeq
Indeed, \eqref{fsx12} yields
\[
\frac{1}{2}\Big( (X_2(t)+ \sigma_2 t) - (X_1(t)+ \sigma_1 t) \Big) \ge \frac{ \sigma_2-  \sigma_1}{4} t >0, \quad t\in (0, T],
\]
which implies \eqref{phidecay}.\\
Thus,
\[
J_2 \le \frac{C}{\delta_*t^2} \int_\R \eta(U|\widetilde{U}) \, dx,
\]
which implies that for any $0<\delta_*<1$ small enough (to be determined below),
		\begin{align}
	\begin{aligned}\label{L2.4}
		-\mathcal{D}(U)\le -\sum_{i=1}^2\delta_i \alpha_m(1-C(\delta_0+\lambda+\e_1+\delta_*))\int_0^1y_i(1-y_i)|\pa_{y_i}w_i|^2\,dy_i + \frac{C}{\delta_* t^2}\int_{\bbr}\eta(U|\tU)\,dx.
		\end{aligned}
	\end{align}

\noindent $\bullet$ {\bf (Conclusion):}
Combining the estimates \eqref{L2.3} and \eqref{L2.4}, we have
\begin{align*}
	\mathcal{B}_1&+\mathcal{B}_2-\widehat{\mathcal{G}_{2}}-\frac{3}{4}\mathcal{D}\\
	&\le \sum_{i=1}^2\delta_i\alpha_m\left((1+C(\delta_0+\lambda+\e_1))\int_0^1|w_i|^2\,dy_i-\frac{3}{4}(1-C_0(\delta_0+\lambda+\e_1+\delta_*))\int_0^1y_i(1-y_i)|\pa_{y_i}w_i|^2\,dy_i\right)\\
	&\quad +C\left(\sum_{i=1}^2 \delta_i^2 \exp(-C \delta_i t) + \frac{1}{\delta_*t^2}  \right) \int_\R \eta (U|\widetilde{U})\,dx.
\end{align*}
Now, choose $\delta_*$ as
\[
\delta_* = \frac{1}{12C_0},
\]
which together with the smallness of $\delta_0,\lambda,\e_1$ yields
\[
C_0(\delta_0+\lambda+\e_1+\delta_*) <\frac{1}{6},
\]
and
\begin{align*}
	\mathcal{B}_1&+\mathcal{B}_2-\widehat{\mathcal{G}_{2}}-\frac{3}{4}\mathcal{D}\\
	&\le  \sum_{i=1}^2\delta_i\alpha_m\left(\frac{9}{8}\int_0^1|w_i|^2\,dy_i-\frac{5}{8}\int_0^1y_i(1-y_i)|\pa_{y_i}w_i|^2\,dy_i\right) \\
	&\quad +C\left(\sum_{i=1}^2 \delta_i \exp(-C \delta_i t) + \frac{1}{ t^2}  \right) \int_\R \eta (U|\widetilde{U})\,dx .
\end{align*}
Then, using Lemma \ref{lem-poin} with the identity:
\[
\int_0^1 |w-\bar w|^2 dy = \int_0^1w^2 dy -{\bar w}^2,\qquad \bar w:=\int_0^1 w dy,
\]
we have
\begin{align*}
	\mathcal{B}_1+\mathcal{B}_2-\widehat{\mathcal{G}_{2}}-\frac{3}{4}\mathcal{D}
	&\le \sum_{i=1}^2\left[-\frac{\delta_i\alpha_m}{8}\int_0^1|w_i|^2\,dy_i +\frac{5\delta_i\alpha_m}{4}\left(\int_0^1w_i\,dy_i\right)^2\right]\\
	&\quad +C\left(\sum_{i=1}^2 \delta_i^2 \exp(-C \delta_i t) + \frac{1}{t^2}  \right) \int_\R \eta (U|\widetilde{U})\,dx.
\end{align*}
Finally, using \eqref{L2.2} with the choice $M=\frac{5}{4}\sigma_m^4\alpha_m$, we have
\begin{align*}
	\begin{aligned}
	&-\sum_{i=1}^2\frac{\delta_i}{2M}|\dot{X}_i|^2+\mathcal{B}_1+\mathcal{B}_2-\mathcal{G}_2-\frac{3}{4}\mathcal{D}\\
	&\le \sum_{i=1}^2\Bigg[-\frac{\delta_i\alpha_m}{16}\int_0^1|w_i|^2\,dy_i+\sigma_i\int_{\bbr}(a_i)^{X_i}_x\frac{1+\gamma}{3\gamma^2}p(\widetilde{v})^{-\frac{1}{\gamma}-2}(p(v)-p(\widetilde{v}))^3\,dx\\
	&\hspace{2cm} -\frac{\sigma_i}{2\gamma}\int_{\bbr}(a_i)^{X_i}_x \left(p(\widetilde{v})^{-\frac{1}{\gamma}-1}-p(\widetilde{v}^{X_i}_i)^{-\frac{1}{\gamma}-1}\right)|p(v)-p(\widetilde{v})|^2\,dx\Bigg]\\
	&\quad+C\left(\sum_{i=1}^2 \delta_i^2 \exp(-C \delta_i t) + \frac{1}{t^2}  \right) \int_\R \eta (U|\widetilde{U})\,dx,
	\end{aligned}
\end{align*}
which implies 
\begin{align}
	\begin{aligned}\label{est-I1}
	\mathcal{R}_1
	&\le \sum_{i=1}^2\Bigg[-C_1 \int_\R |(\widetilde{v}_i)^{X_i}_x| |\phi_i(p(v)-p(\widetilde{v})) |^2 \, dx
	+C \int_\R |(a_i)^{X_i}_x| |p(v)-p(\widetilde{v})|^3 \, dx\\
	&\hspace{1.5cm}+C \int_{\bbr} |(a_i)^{X_i}_x| |\tv-\tvi^{X_i}| |p(v)-p(\widetilde{v})|^2\,dx\Bigg]\\
	&\quad+C\left(\sum_{i=1}^2 \delta_i^2\exp(-C \delta_i t) + \frac{1}{t^2}  \right) \int_\R \eta (U|\widetilde{U})\,dx.
	\end{aligned}
\end{align}

\subsection{Estimate of the remaining terms}\label{sec:4.6}
We substitute \eqref{est-I1} to \eqref{est} and use the Young's inequality
\beq\label{2young}
\sum_{i=1}^2\left(\dot{X}_i\sum_{j=3}^6Y_{ij}\right)\le \sum_{i=1}^2\frac{\delta_i}{4M}|\dot{X}_i|^2 +\sum_{i=1}^2\frac{C}{\delta_i}\sum_{j=3}^6|Y_{ij}|^2
\eeq
to have
\begin{align}
	\begin{aligned}\label{est-2}
	&\frac{d}{dt}\int_{\bbr}a\eta(U|\widetilde{U})\,dx \\
	&\leq-C_1\mathcal{G}^S + \mathcal{K}_1+\mathcal{K}_2 \\
	&\quad +C\left(\sum_{i=1}^2 \delta_i^2 \exp(-C \delta_i t) + \frac{1}{t^2}  \right) \int_\R \eta (U|\widetilde{U})\,dx \\
	&\quad  -\sum_{i=1}^2\frac{\delta_i}{4M}|\dot{X}_i|^2 + \sum_{i=1}^2 \frac{C}{\delta_i}\sum_{j=3}^6|Y_{ij}|^2 +\sum_{i=3}^5\mathcal{B}_i +\mathcal{S}_1+\mathcal{S}_2-\mathcal{G}_1-\frac{1}{4}\mathcal{D},
	\end{aligned}
\end{align}
where
\begin{align*}
&\mathcal{G}^S:=\sum_{i=1}^2 \int_\R |(\widetilde{v}_i)^{X_i}_x| |\phi_i(p(v)-p(\widetilde{v})) |^2 dx \\
&\mathcal{K}_1:=\sum_{i=1}^2 \underbrace{C \int_\R |(a_i)^{X_i}_x| |p(v)-p(\widetilde{v})|^3 dx}_{=: \mathcal{K}_{i1}} \\
 &\mathcal{K}_2:=\sum_{i=1}^2 \underbrace{C \int_\R |(a_i)^{X_i}_x| |\tv-\tvi^{X_i}| |p(v)-p(\widetilde{v})|^2\,dx}_{=: \mathcal{K}_{i2}} 
\end{align*}

In what follows,  to control the remaining terms,  we will use the good terms $\cG_1$,$\cG^S$ and the diffusion term $\cD$.\\

\noindent $\bullet$ {\bf (Estimate of $\mathcal{K}_1$):}
 To control the term $\mathcal{K}_{i1}$ by the above good terms, we localize it via $\phi_i$ as
\begin{equation}\label{Ki1}
\mathcal{K}_{i1}
\leq C \frac{\lambda}{\delta_i} \int_{\R} |(\widetilde{v}_i)^{X_i}_x| \phi_i |p(v)-p(\widetilde{v})|^3  \,dx+C \frac{\lambda}{\delta_i} \int_{\R} |(\widetilde{v}_i)^{X_i}_x| (1-\lpi) |p(v)-p(\widetilde{v})|^3  \,dx.
\end{equation}
Using the Gagliardo-Nirenberg interpolation inequality and $\delta_i\ll\lambda\le C\sqrt{\delta_i}$, the first term is controlled by the good terms as
\begin{align*}
&C \frac{\lambda}{\delta_i} \int_{\R} |(\widetilde{v}_i)^{X_i}_x| \phi_i |p(v)-p(\widetilde{v})|^3  \,dx\\
&\leq C \frac{\lambda}{\delta_i} \norm{\pv-p(\tv)}^2_{L^\infty(\R)} \sqrt{\int_{\R} |(\tv_i)^{X_i}_x| |\lpi(\pv-p(\tv))|^2 \,dx}\sqrt{\int_{\R} |(\tv_i)^{X_i}_x| \,dx} \\
&\leq C \frac{\lambda}{\sqrt{\delta_i}} \norm{\partial_x(\pv-p(\tv))}_{L^2(\R)}\norm{\pv-p(\tv)}_{L^2(\R)} \sqrt{\int_{\R} |(\tv_i)^{X_i}_x| |\lpi(\pv-p(\tv))|^2 \,dx}\\
&\leq C \varepsilon_1  \norm{\partial_x(\pv-p(\tv))}_{L^2(\R)}\sqrt{\int_{\R} |(\tv_i)^{X_i}_x| |\lpi(\pv-p(\tv))|^2 \,dx} \\
&\leq C \varepsilon_1  \norm{\partial_x(\pv-p(\tv))}^2_{L^2(\R)}+C \varepsilon_1 \int_{\R} |(\widetilde{v}_i)^{X_i}_x| | \phi_i(p(v)-p(\widetilde{v}))|^2 \, dx 
\le \frac{1}{48}(\mathcal{D}+C_1 \mathcal{G}^S )
\end{align*}
Using Lemma \ref{lem:shock-interact-2}, the second term in \eqref{Ki1} is estimated as
\begin{align*}
C \frac{\lambda}{\delta_i} \int_{\bbr} |(\widetilde{v}_i)^{X_i}_x| (1-\lpi) |p(v)-p(\widetilde{v})|^3 \, dx &\leq C \eps_1\lambda \delta_i \exp(-C \delta_i t) \int_\R |\pv-p(\tv)|^2 \, dx\\
&\leq C \eps_1\lambda \delta_i \exp(-C \delta_i t)\int_\R  \eta (U|\widetilde{U}) \, dx.
\end{align*}
Combining the estimates for $i=1,2$, we derive
\begin{align*}
\mathcal{K}_1\leq  \frac{1}{24}(\mathcal{D}+C_1 \mathcal{G}^S ) + C\sum_{i=1}^2  \eps_1 \lambda \delta_i \exp(-C \delta_i t)\int_\R  \eta (U|\widetilde{U}) \, dx.
\end{align*}

\noindent $\bullet$ {\bf (Estimate of $\mathcal{K}_2$):}
To estimate $\mathcal{K}_{2}$, we use Lemma \ref{lem:shock-interact-1} to obtain
\begin{align*}
\mathcal{K}_{i2}
&=C\int_\R |(\ai)^{X_i}_x| |\tv-\tv^{X_i}_i| |\pv-p(\tv)|^2 \, dx\\
&\leq C \frac{\lambda}{\delta_i} \int_\R |(\tv_i)^{X_i}_x||\tv-\tv^{X_i}_i| |\pv-p(\tv)|^2 dx\\
&\leq C\lambda \delta_1 \delta_2 \exp(-C \min (\delta_1, \delta_2) t) \int_\R \eta(U |\widetilde{U}) \, dx,
\end{align*}
which completes the estimate on $\mathcal{K}_2$:
\[\mathcal{K}_2\le C\lambda\delta_1\delta_2\exp(-C\min\{\delta_1,\delta_2\}t)\int_{\R}\eta(U|\tU)\,dx.\]

\noindent $\bullet$ {\bf (Estimate of $\frac{C}{\delta_i} |Y_{ij}|^2$ for $i=1,2, j=3,\ldots, 6$):}
Using \eqref{viscous-shock-h}$_2$ and \eqref{inta}, we estimate $Y_{i3}$ as
\begin{align*}
|Y_{i3}|&=\left| \frac{1}{\sigma_i} \int_\R a p(\tv_i^{X_i})_x \left(h-\widetilde{h}-\frac{\pv-p(\tv)}{\sigma_i} \right) \, dx \right| \\
&\leq C \frac{\delta_i}{\lambda} \int_\R \left|(\ai)^{X_i}_x\right| \left| h-\widetilde{h}-\frac{\pv-p(\tv)}{\sigma_i} \right| \, dx \leq C \frac{\delta_i}{\sqrt{\lambda}} \sqrt{\cG_1},
\end{align*}
which gives
\[
\frac{C}{\delta_i} |Y_{i3}|^2 \le\frac{C\delta_i}{\lambda}\mathcal{G}_{1}.
\]
Using Lemma \ref{lem:shock-interact-1}, we control $Y_{i4}$ as
\begin{align*}
|Y_{i4}| 
&\leq  C \int_\R |\tv-\tv^{X_i}_i| |(\tv_i)^{X_i}_x| |\pv-p(\tv)| \, dx\\
&\leq C\sqrt{\int_\R (|\tv-\tv^{X_1}_1||(\tv_1)^{X_1}_x|)^2 \, dx} \sqrt{\int_\R \eta(U |\widetilde{U}) \, dx} \\
&\le  C\sqrt{\delta_i} \delta_1 \delta_2 \exp(-C \min (\delta_1, \delta_2) t)  \sqrt{\int_\R \eta(U |\widetilde{U}) \, dx},
\end{align*}
which yields
\[
\frac{C}{\delta_i} |Y_{i4}|^2 \le  C\delta_1^2 \delta_2^2 \exp(-C \min (\delta_1, \delta_2) t) \int_\R \eta(U |\widetilde{U}) \, dx.
\]
For $Y_{i5}$,  we first estimate $h-\widetilde{h}$ in terms of $u-\tu$ and $v-\tv$ as follows.  Using $\tv_x=(\tv_1)^{X_1}_x+(\tv_2)^{X_2}_x$ and $C^{-1}\le v, \tv^{X_1}_1, \tv^{X_2}_2 \le C$,  we have
\begin{align}
\begin{aligned}\label{h-est}
|h-\widetilde{h}| &\leq |u-\tu| + |(\ln v)_x+(\ln \tv^{X_1}_1)_x - (\ln \tv^{X_2}_2)_x|\\
&= |u-\tu|+\left|\frac{v_x}{v}-\frac{(\tv_1)^{X_1}_x}{\tv^{X_1}_1}-\frac{(\tv_2)^{X_2}_x}{\tv^{X_2}_2}\right|\\
&\leq |u-\tu|+\left|\frac{v_x-\tv_x}{v}\right|+\left|\left(\frac{1}{v}-\frac{1}{\tv}\right)\tv_x\right|+\left|\left(\frac{1}{\tv}-\frac{1}{\tv^{X_1}_1}\right)(\tv_1)^{X_1}_x\right|+\left|\left(\frac{1}{\tv}-\frac{1}{\tv^{X_2}_2}\right)(\tv_2)^{X_2}_x\right|\\
&\leq |u-\tu|+C(|(v-\tv)_x|+|v-\tv||\tv_x| +|\tv-\tv^{X_1}_1||(\tv_1)^{X_1}_x|+|\tv-\tv^{X_2}_2||(\tv_2)^{X_2}_x|).
\end{aligned}
\end{align}
Then, Lemma \ref{lem:shock-interact-1} implies
\begin{align*}
\lVert h-\widetilde{h} \rVert_{L^2(\R)} 
&\leq C \Bigg( \norm{u-\tu}_{L^2(\R)}+\norm{v-\tv}_{H^1(\R)}\\
&\hspace{1.5cm}+\sqrt{\int_\R (|\tv-\tv^{X_1}_1||(\tv_1)^{X_1}_x|)^2 \, dx}+\sqrt{\int_\R (|\tv-\tv^{X_2}_2|
|(\tv_2)^{X_2}_x|)^2 \, dx} \Bigg)\\
&\leq C \left( \norm{u-\tu}_{L^2(\R)}+\norm{v-\tv}_{H^1(\R)}+C \delta_1^{3/2} \delta_2+C \delta_1 \delta_2^{3/2} \right),
\end{align*}
and therefore,
\beq\label{normh}
 \lVert h-\widetilde{h} \rVert_{L^\infty(0,T;L^2(\R))} \leq C(\varepsilon_1+C \delta_1^{3/2} \delta_2+C \delta_1 \delta_2^{3/2}).
 \eeq
Hence, we estimate $Y_{i5}$ as  
\begin{align*}
|Y_{i5}|^2 
&\leq C  \left(\frac{\sigma_i}{2}\int_\R (\ai)^{X_i}_x \left| h-\widetilde{h}-\frac{\pv-p(\tv)}{\sigma_i}\right|^2 \, dx\right)\left(\int_\R (\ai)^{X_i}_x\left| h-\widetilde{h}+\frac{\pv-p(\tv)}{\sigma_i}\right|^2 \, dx\right) \\
&\leq C \cG_{1i} \norm{(\ai)^{X_i}_x}_{L^\infty} \left( \lVert{h-\widetilde{h} \rVert}_{L^\infty(0,T;L^2(\R))}+\lVert{v-\widetilde{v} \rVert}_{L^\infty(0,T;L^2(\R))} \right)^2\\
&\leq C\lambda\delta_i (\e_1+C\delta_1\delta_2)^2 \mathcal{G}_{1},
\end{align*}
which yields
\[
\frac{C}{\delta_i} |Y_{i5}|^2 \le C\lambda (\e_1+C\delta_1\delta_2)^2 \mathcal{G}_{1}.
\]
Finally, using Lemma \ref{lem-rel-quant} (3), $\lambda\le C\sqrt{\delta_i}$ and Lemma \ref{lem:shock-interact-2}, we estimate $Y_{i6}$ as
\begin{align*}
\frac{C}{\delta_i} |Y_{i6}|^2 &\leq \frac{C}{\delta_i} \left( \int_\R |(\ai)^{X_i}_x| |\pv-p(\tv)|^2 \, dx\right)^2 \leq \frac{C \lambda^2}{\delta_i^3} \left( \int_\R |(\tvi)^{X_i}_x| |\pv-p(\tv)|^2 \, dx\right)^2 \\
&\leq \frac{C \lambda^2} {\delta_i} \norm{\pv-p(\tv)}^2_{L^2(\R)}  \int_\R |(\tv_i)^{X_i}_x| |\pv-p(\tv)|^2 \, dx \\
&\leq C \varepsilon_1^2  \int_\R |(\tv_i)^{X_i}_x| |\pv-p(\tv)|^2 \, dx\\
&= C \varepsilon_1^2 \int_\R |(\tv_i)^{X_i}_x| | \lpi(\pv-p(\tv))|^2 \, dx+ C \varepsilon_1^2 \int_\R |(\tv_i)^{X_i}_x|(1-\lpi^2) | \pv-p(\tv)|^2 \, dx\\
&\leq \frac{C_1}{48} \cG^S_i+C \varepsilon_1^2 \int_\R |(\tv_i)^{X_i}_x| (1-\phi_i) |\pv-p(\tv)|^2 \, dx\\
&\le\frac{C_1}{48} \cG^S_i+C \varepsilon_1^2 \delta_i^2 \exp(-C \delta_i t) \int_\R \eta(U|\widetilde{U}) \, dx.
\end{align*}

Combining all the estimates for $Y_{ij}$ and use the smallness of the parameters, we conclude that
\begin{align*}
	\sum_{i=1}^2\left(\frac{C}{\delta_i}\sum_{j=3}^6|Y_{ij}|^2\right)&\le \sum_{i=1}^2\bigg(\frac{C\delta_i}{\lambda}\mathcal{G}_{1} +C\lambda(\e_1+C\delta_1\delta_2 )^2\mathcal{G}_{1}\\
	&\quad+\frac{C_1}{48}\mathcal{G}^S 
	+C\Big(\delta_1^2 \delta_2^2 \exp(-C \min (\delta_1, \delta_2) t) + \e_1^2\delta_i^2\exp(-C\delta_it) \Big)\int_{\R}\eta(U|\tU)\,dx\\
	&\le \frac{1}{4}\mathcal{G}_1 + \frac{C_1}{24}\mathcal{G}^S\\
	&\quad+C \Big(\delta_1^2 \delta_2^2 \exp(-C \min (\delta_1, \delta_2) t) + \e_1^2\sum_{i=1}^2\delta_i^2\exp(-C\delta_it) \Big)\int_{\R}\eta(U|\tU)\,dx.
\end{align*}

\noindent $\bullet$ {\bf (Estimate of $\mathcal{B}_3$): }
Recall $\mathcal{B}_3:= \cB_{13}+\cB_{23}$ where
\[
\cB_{i3}: =-\int_{\bbr}(a_i)^{X_i}_x\frac{p(v)-p(\widetilde{v})}{\gamma p(v)}\pa_x(p(v)-p(\widetilde{v}))\,dx.
\]
As done before, using Lemma \ref{lem:shock-interact-2}, we estimate $\cB_{i3}$ as
\begin{align*}
|\cB_{i3}| 
&\leq \frac{1}{48} \cD+C\left(\frac{\lambda}{\delta_i}\right)^2 \int_\R |(\tv_i)^{X_i}_x|^2 |\pv-p(\tv)|^2 \, dx\\
&\leq \frac{1}{48} \cD+C\lambda^2 \int_\R |(\tv_i)^{X_i}_x| |\phi_i(\pv-p(\tv))|^2 \, dx +C \lambda^2 \int_\R |(\tv_i)^{X_i}_x|(1-\lpi) |\pv-p(\tv)|^2 \, dx\\
&\le\frac{1}{48}(\cD+C_1 \cG^S)+C \lambda^2 \delta_i^2 \exp(-C \delta_i t) \int_\R \eta(U | \widetilde{U}) \, dx.
\end{align*}
Therefore, we have
\begin{align*}
\cB_3 \leq \frac{1}{24}(\cD +C_1 \cG^S)+C\lambda^2 \sum_{i=1}^2 \delta_i^2 \exp(-C \delta_i t)\int_\R \eta(U | \widetilde{U}) \, dx.
\end{align*}

\noindent $\bullet$ {\bf(Estimate of $\cB_4$): }
Recall $\mathcal{B}_4:= \cB_{14}+\cB_{24}$ where
\[
\cB_{i4}: =-\int_{\bbr}(a_i)^{X_i}_x(p(v)-p(\widetilde{v}))^2\frac{\pa_xp(\widetilde{v})}{\gamma p(v)p(\widetilde{v})}\,dx.
\]
Likewise, we estimate $\cB_{i4}$ by using Lemma \ref{lem:shock-interact-2} with \eqref{inta} and $|\partial_x p(\tv)| \leq C \left(|(\tv_1)^{X_1}_x|+|(\tv_2)^{X_2}_x| \right)$ as
\begin{align*}
|\cB_{i4}| 
&\leq  C \lambda \delta_i \sum_{j=1}^2 \int_\R |(\tv_j)^{X_j}_x||\pv-p(\tv)|^2 \, dx\\
 &\le C \lambda \delta_i \cG^S + C \lambda \delta_i \sum_{j=1}^2 \int_\R |(\tv_j)^{X_j}_x| |1-\phi_j| |\pv-p(\tv)|^2 \, dx\\
 &\leq \frac{C_1}{12} \cG^S+ C \lambda \delta_i \sum_{j=1}^2  \delta_j^2 \exp(-C \delta_j t) \int_\R \eta(U | \widetilde{U}) \, dx .
\end{align*}
Therefore, we bound $\mathcal{B}_4$ as
\begin{align*}
|\cB_4| \leq \frac{C_1}{6} \cG^S+ C\sum_{j=1}^2  \delta_j^2 \exp(-C \delta_j t) \int_\R \eta(U | \widetilde{U}) \, dx .
\end{align*}

\noindent $\bullet${\bf (Estimate for $\mathcal{B}_5$): }
As in $\mathcal{B}_{i3}$, we have
\begin{align*}
|\cB_{5}| 
&\leq \frac{1}{48} \cD+C \int_\R \left( |(\tv_1)^{X_1}_x|^2+|(\tv_2)^{X_2}_x|^2 \right) |\pv-p(\tv)|^2 \, dx\\
& \leq \frac{1}{48} \cD+\frac{C_1}{12}\cG^S+C(\delta_1^2 \exp(-C \delta_1 t)+\delta_2^2 \exp(-C \delta_2 t)) \int_\R \eta (U |\widetilde{U}) \, dx.
\end{align*}

\noindent $\bullet${\bf (Estimate for the interaction terms $\mathcal{S}_1$ and $\mathcal{S}_2$): }
We first estimate $E_1$ and $E_2$. We use $(\tv)_x=(\tv_1)^{X_1}_x+(\tv_2)^{X_2}_x$ to estimate $E_1$ as
\begin{align*}
E_1&=\partial_x \left( \frac{(\tv)_x}{\tv}-\frac{(\tv_1)^{X_1}_x}{\tv^{X_1}_1}-\frac{(\tv_2)^{X_2}_x}{\tv^{X_2}_2}\right)=\partial_x \left( \left(\frac{1}{\tv}-\frac{1}{\tv_1^{X_1}}\right)(\tv_1)^{X_1}_x+\left(\frac{1}{\tv}-\frac{1}{\tv^{X_2}_2}\right)(\tv_2)^{X_2}_x\right)\\
&=\left(\frac{1}{\tv}-\frac{1}{\tv^{X_1}_1}\right)(\tv_1)^{X_1}_{xx}+\left(\frac{1}{\tv}-\frac{1}{\tv^{X_2}_2}\right)(\tv_2)^{X_2}_{xx}\\
&\quad +\left(\frac{1}{(\tv^{X_1}_1)^2}-\frac{1}{(\tv)^2}\right)|(\tv_1)^{X_1}_x|^2+\left(\frac{1}{(\tv^{X_2}_2)^2}-\frac{1}{(\tv)^2}\right)|(\tv_2)^{X_2}_x|^2-2\frac{(\tv_1)^{X_1}_x(\tv_2)^{X_2}_x}{(\tv)^2}.
\end{align*}
Therefore, thanks to \eqref{second-order}, $E_1$ is bounded as
\begin{align*}
|E_1| &\leq C \delta_1 |(\tv_1)^{X_1}_x| |\tv-\tv^{X_1}_1|+C \delta_2 |(\tv_2)^{X_2}_x| |\tv-\tv^{X_2}_2|+C|(\tv_1)^{X_1}_x||(\tv_2)^{X_2}_x|.
\end{align*}
Similarly, we estimate $E_2$ as
\begin{align*}
|E_2| &\leq |p'(\tv)-p'(\tv^{X_1}_1)||(\tv_1)^{X_1}_x|+|p'(\tv)-p'(\tv^{X_2}_2)||(\tv_2)^{X_2}_x| \\
&\leq C|(\tv_1)^{X_1}_x||\tv-\tv_1^{X_1}|+C|(\tv_2)^{X_2}_x||\tv-\tv_2^{X_2}|.
\end{align*}
Therefore, the above estimates yield
\begin{align*}
|\cS_1|+|\cS_2| 
&\leq C \sum_{i=1}^2 \bigg(\int_\R |(\tv_i)^{X_i}_x||\tv-\tv^{X_i}_i| \big( |\pv-p(\tv)| + |h-\widetilde{h} | \big) \, dx  \bigg) \\ 
&\quad +C \int_\R |(\tv_1)^{X_1}_x||(\tv_2)^{X_2}_x| |\pv-p(\tv)| \, dx.
\end{align*}
Using \eqref{normh}, Lemma \ref{lem:shock-interact-1} and the assumption \eqref{perturbation_small} with \eqref{smp1}, we have
\begin{align*}
|\cS_1|+|\cS_2| 
&\leq C \sum_{i=1}^2\bigg( (\| v-\tv \|_{L^\infty(0,T;L^2(\R))} + \lVert h-\widetilde{h} \rVert_{L^\infty(0,T;L^2(\R))} )\sqrt{ \int_\R |(\tv_i)^{X_i}_x|^2|\tv-\tv^{X_i}_i|^2dx }  \bigg)\\
&\quad  +C\eps_1 \int_\R |(\tv_1)^{X_1}_x||(\tv_2)^{X_2}_x| \, dx\\
&\le C\eps_1 \delta_1 \delta_2( \delta_1^{1/2} +\delta_2^{1/2}) \exp(-C \min(\delta_1,\delta_2) t)+ C\delta_1^2 \delta_2^2 (\delta_1+ \delta_2)  \exp(-C \min(\delta_1,\delta_2) t)\\
&\quad + C\eps_1\delta_1\delta_2  \exp(-C \min(\delta_1,\delta_2) t) .
\end{align*}
Hence we finally estimate $\cS_1$ and $\cS_2$ as
\beq\label{s1s2}
	|\cS_1|+|\cS_2|\le C\delta_1\delta_2  \exp(-C \min(\delta_1,\delta_2) t).	
	\eeq

\subsection{Estimate in small time} \label{sec:sest}
Note that the estimate  \eqref{est-I1} on $\mathcal{R}_1$ contains the coefficient $\frac{1}{t^2}$, which is not integrable near $t=0$.  Thus, in order to get the desired result, we would find a rougher estimate in a short time. To this end, we return to the previous right-hand side $\mathcal{R}$ in \eqref{est}:
\begin{align*}
	\mathcal{R}=-\sum_{i=1}^2\frac{\delta_i}{M}|\dot{X}_i|^2 + \sum_{i=1}^2\left(\dot{X_i}\sum_{j=3}^6Y_{ij}\right) +\sum_{i=1}^5\mathcal{B}_i +\mathcal{S}_1+\mathcal{S}_2-\mathcal{G}_1-\mathcal{G}_2-\mathcal{D}.
\end{align*}
By the Young's inequality \eqref{2young}, we first have
\begin{align*}
	\mathcal{R}+\sum_{i=1}^2\frac{\delta_i}{4M}|\dot{X}_i|^2 + \mathcal{G}_1 + \mathcal{D} + \mathcal{G}^S \le  \sum_{i=1}^2\left(\frac{C}{\delta_i}\sum_{j=3}^6|Y_{ij}|^2\right)+\sum_{i=1}^5 \mathcal{B}_i+\mathcal{S}_1+\mathcal{S}_2 + \mathcal{G}^S.
\end{align*}
Using \eqref{perturbation_small} and Lemma \ref{lem:shock-est} with \eqref{inta}, we have
\begin{align*}
\sum_{j=3}^6|Y_{ij}| &\le  C  \norm{(\tv_i)^{X_i}_x}_{L^2(\R)}  (\|h-\widetilde{h}\|_{L^2(\R)}+\norm{\pv-p(\tv)}_{L^2(\R)}) \\
& +  \norm{(a_i)^{X_i}_x}_{L^\infty(\R)}  (\|h-\widetilde{h}\|_{L^2(\R)}^2+\norm{\pv-p(\tv)}_{L^2(\R)}^2) \\
&\le C \delta_i \eps_1,
\end{align*}
which yields
\[
 \sum_{i=1}^2\left(\frac{C}{\delta_i}\sum_{j=3}^6|Y_{ij}|^2\right) \le C\eps_1^2  \sum_{i=1}^2 \delta_i.
\]
Likewise, we have
\[
\sum_{i=1}^5 \mathcal{B}_i \le  C\sum_{i=1}^2 \Big(\norm{(a_i)^{X_i}_x}_{L^\infty(\R)} +  \norm{(\tv_i)_x}_{L^\infty(\R)} \Big) \norm{v-\tv}_{H^1(\R)}^2 \le C \eps_1^2  \sum_{i=1}^2\delta_i,
\]	
and
\[
 \mathcal{G}^S \le C\sum_{i=1}^2 \norm{(\tv_i)^{X_i}_x}_{L^\infty(\R)} \norm{\pv-p(\tv)}_{L^2(\R)}^2 \le C\eps_1^2\sum_{i=1}^2 \delta_i^2.
\]
Hence, the above estimates and \eqref{s1s2} provides a rough bound: for any $\delta_1, \delta_2\in (0,\delta_0)$,
\beq\label{sest}
\mathcal{R}+\sum_{i=1}^2\frac{\delta_i}{4M}|\dot{X}_i|^2 + \mathcal{G}_1 + \mathcal{D} + \mathcal{G}^S \le C\delta_0,\quad t>0.
\eeq

\subsection{Proof of Lemma \ref{lem:rel-ent}} We here complete the proof of Lemma \ref{lem:rel-ent}. First of all, from \eqref{sest} with \eqref{est}, we have a rough estimate for a short time $t\le 1$ as follows: 
\begin{align*}
\frac{d}{dt}\int_{\R}a\eta(U|\tU)\,dx&\le C\delta_0,
\end{align*}
which implies
\begin{equation}\label{est-short-time}
	\int_{\R} a\eta(U|\tU)\,d x\Bigg|_{t=1} +\int_0^1\bigg(\sum_{i=1}^2\frac{\delta_i}{4M}|\dot{X}_i|^2 + \mathcal{G}_1 + \mathcal{D} + \mathcal{G}^S \bigg) dt \le \int_{\R} a\eta(U|\tU)\,d x\Bigg|_{t=0} + C\delta_0.
\end{equation}
On the other hand, for $t\ge 1$, we combine all the estimates in Section \ref{sec:4.5} and Section \ref{sec:4.6} to derive

\begin{align*}
	\frac{d}{dt}&\int_{\R}a\eta(U|\tU)\,dx+\sum_{i=1}^2\frac{\delta_i}{4M}|\dot{X}_i|^2+\frac{1}{2}\mathcal{G}_1 +\frac{C_1}{2}\mathcal{G}^S+\frac{1}{8}\mathcal{D}\\
	&\le C\left(\sum_{i=1}^2\delta_i \exp(-C\delta_i t)+ \delta_1\delta_2\exp(-C\min(\delta_1,\delta_2)t)+\frac{1}{t^2}\right)\int_{\R}a\eta(U|\tU)\,dx\\
	&\quad +C\delta_1\delta_2\exp(-C\min(\delta_1,\delta_2)t).
\end{align*}
Hence, we use Gr\"onwall inequality to conclude that for all $t\ge 1$,
\begin{align}
	\begin{aligned}\label{est-long-time}
	\int_{\R}a&\eta(U(t,x)|\tU(t,x))\,dx+\int_1^t\left(\sum_{i=1}^2\frac{\delta_i}{4M}|\dot{X}_i|^2+\frac{1}{2}\mathcal{G}_1 +\frac{C_1}{2}\mathcal{G}^S+\frac{1}{8}\mathcal{D}\right)\,ds\\
	&\le \left(\int_{\R}a\eta(U|\tU)\,dx\Bigg|_{t=1}+\frac{C\delta_1\delta_2}{\min(\delta_1,\delta_2)}\right)\\
	&\quad \times\exp\left(\int_1^t \sum_{i=1}^2\delta_i \exp(-C\delta_i s)+\delta_1\delta_2\exp(-C\min(\delta_1,\delta_2)s)+\frac{1}{s^2}\,ds\right)\\
	&\le C\int_{\R}a\eta(U|\tU)\,dx\Bigg|_{t=1}+C \max(\delta_1,\delta_2).
	\end{aligned}
\end{align}
Finally, we combine the estimates  \eqref{est-short-time} and \eqref{est-long-time} to conclude that
\begin{align*}
	\int_{\R}a&\eta(U(t,x)|\tU(t,x))\,dx+\int_0^t \left(\sum_{i=1}^2 \delta_i|\dot{X}_i|^2 + \mathcal{G}_1+\mathcal{G}^S+\mathcal{D}\right)\,ds\\
	&\le C\int_{\R}a(0,x)\eta(U_0(x)|\tU(0,x))\,dx +C\delta_0 ,
\end{align*}
which together with $\frac{1}{2}\le a\le 1$, $D(U)\le C\mathcal{D}(U)$ and $G_1(U)\le C\mathcal{G}_1(U)$, completes the proof of Lemma \ref{lem:rel-ent}.

\section{Proof of Proposition \ref{prop:H1-estimate}}\label{sec:5}
\setcounter{equation}{0}

In this section, we complete the proof of Proposition \ref{prop:H1-estimate}. 
\subsection{Zeroth-order estimate}
We first estimate the zeroth-order term of $\norm{u-\tu}_{L^2 (\R)}$.
\begin{lemma}Under the hypotheses of Proposition 3.2,  there exists a positive constant $C>0$, that is independent of $\delta_1,\delta_2,\varepsilon_1,T$, such that for all $t \in (0,T]$,
\begin{equation} \label{temp 0}
\begin{aligned}
&\norm{v-\tv}_{H^1(\R)}^2+\norm{u-\tu}_{L^2(\R)}^2+\int_0^t \sum_{i=1}^2 \delta_i |\dot{X}_i|^2 \,ds +\int_0^t \left( \mathcal{G}^S+D+D_1 \right) \,ds\\
&\qquad \leq C \left( \norm{v_0-\tv(0,\cdot)}_{H^1(\R)}^2+\norm{u_0-\tu(0,\cdot)}_{L^2(\R)}^2 \right)+C \delta_0^{1/2}
\end{aligned}
\end{equation}
where $\mathcal{G}^S$ and $D$ are the good terms of Lemma \ref{lem:rel-ent}, and
\begin{equation}\label{D1}
{D}_1:=\int_\R|(u-\tu)_x|^2 \,dx.
\end{equation}
\end{lemma}
\begin{proof}
As in Section \ref{sec:re}, the system  \eqref{eq:NS} can be written as:
\[U_t +A(U)_x = \pa_x(M(U)\pa_x  D \eta(U)),\]
where 
\[U:=\begin{pmatrix}
	v\\u
\end{pmatrix},\quad A(U):=\begin{pmatrix}
-u\\p(v)
\end{pmatrix},\quad \eta(U):=\frac{u^2}{2}+Q(v)=\frac{u^2}{2}+\frac{v^{-\gamma+1}}{\gamma-1},\quad M(U):=\begin{pmatrix}
0&0\\0&\frac{1}{v}
\end{pmatrix}.\]
Then, the shifted composite wave $\widetilde{U}=(\tv,\tu)$ satisfies
	\begin{equation} \label{eq:viscous-shock-uv}
	\pa_t\widetilde{U} +A(\widetilde{U})_x=\pa_x(M(\widetilde{U})\pa_x D \eta (\widetilde{U}))-\sum_{i=1}^2\dot{X}_i(\widetilde{U}_i)^{X_i}_x +\begin{pmatrix}
		0\\E_2+E_3
	\end{pmatrix},
	\end{equation}
where 
\[E_2:=p(\widetilde{v})_x-\sum_{i=1}^2p(\widetilde{v}^{X_i}_i)_x,\quad 
E_3:=-\left(\frac{\tu_x}{\tv}\right)_x+\sum_{i=1}^2 \left( \frac{(\tu_i)^{X_i}_x}{\tv^{X_i}_i}\right)_x.\]
Then, the relative entropy method (see e.g. \cite[Lemma 5.1]{KVW3}) implies that
\[\frac{d}{dt} \int_\R \eta(U(t,x)|\tU(t,x)) \,dx=\sum_{i=1}^2(\dot{X}_i(t)\mathcal{Y}_i(t))+\sum_{i=1}^6 \mathcal{I}_i(U),\]
where
\begin{align*}
\mathcal{Y}_i &:=\int_\R (\tU_i)^{X_i}_x D^2 \eta(\tU)(U-\tU) \,  dx=-\int_\R p'(\tv)(\tv_i)^{X_i}_x(v-\tv) \,  dx+\int_\R (\tu_i)^{X_i}_x (u-\tu) \,  dx=:\mathcal{Y}_{i1}+\mathcal{Y}_{i2},\\
\mathcal{I}_1 &:=-\int_\R \partial_x G(U;\tU) \, dx=-\int_\R \partial_x \left( (p-\tp)(u-\tu) \right) \, dx=0,\\
\mathcal{I}_2 &:=-\int_\R (\partial_x  D\eta(\tU))A(U|\tU) \, dx=-\sum_{i=1}^2 \int_\R  (\tu_i)^{X_i}_x p(v|\tv) \,  dx=:\mathcal{I}_{21}+\mathcal{I}_{22},\\
\mathcal{I}_3 &:=\int_\R \left( D \eta(U)-D \eta(\tU)\right) \partial_x \left(M(U) \partial_x ( D \eta(U)-D \eta(\tU) )\right) \,  dx=-\underbrace{\int_\R \frac{1}{v} |\partial_x (u-\tu)|^2 \, dx}_{=:\mathbf{D}_1},\\
\mathcal{I}_4 &:=\int_\R \left( D\eta(U)- D\eta (\tU) \right) \partial_x \left( ( M(U)-M(\tU) ) \partial_x D \eta (\tU) \right) \, dx=\int_\R (u-\tu)\left(\left(\frac{1}{v}-\frac{1}{\tv}\right) \tu_x \right)_x \,  dx,\\
\mathcal{I}_5 &:=\int_\R (D \eta) (U | \tU) \partial_x \left( M(\tU) \partial_x D \eta (\tU) \right) \,  dx=\int_\R
\begin{pmatrix}
	-p(v|\tv)\\0
\end{pmatrix}\cdot
\begin{pmatrix}
	0 \\ (\frac{\tu_x}{\tv})_x 
\end{pmatrix}
\, dx =0,\\
\mathcal{I}_6 &:=-\int_\R D^2 \eta(\tU)(U-\tU)
 \begin{pmatrix} 0 \\ E_2+E_3 \end{pmatrix} \,  dx=-\int_\R (u-\tu)(E_2+E_3) \,dx\\
 &=-\int_\R (u-\tu) \left( \sum_{i=1}^2 p(\tv^{X_i}_i)_x-p(\tv)_x+\sum_{i=1}^2 \left( \frac{(\tu_i)^{X_i}_x}{\tv^{X_i}_i}\right)_x-\left(\frac{\tu_x}{\tv}\right)_x \right) \,dx.
\end{align*}
In the following, we estimate each term above one by one.\\

\noindent $\bullet$ (Estimate of $\mathcal{Y}_{i1}$): 
We use $\phi_1+\phi_2=1$ to have
\[\mathcal{Y}_{11}=-\int_\R p'(\tv)(\tv_1)^{X_1}_x \phi_1 (v-\tv) \,dx-\int_\R p'(\tv)(\tv_1)^{X_1}_x \phi_2 (v-\tv) \,dx.
\]
Since
\[
\left|\int_\R p'(\tv)(\tv_1)^{X_1}_x \phi_1(v-\tv) \,dx\right|^2 \leq C \int_\R | (\tv_1)^{X_1}_x| \,  dx\int_\R |(\tv_1)^{X_1}_x| | \phi_1(v-\tv)|^2 \, dx \leq C \delta_1\mathcal{G}^S,
\]
and (by Lemma \ref{lem:shock-interact-2})
\[\left|\int_\R p'(\tv)(\tv_1)^{X_1}_x \phi_2 (v-\tv) \,  dx \right|^2 \leq C \int_\R \left| (\tv_1)^{X_1}_x \phi_2  \right|^2 dx\int_\R \left| v-\tv\right|^2 dx \leq C \delta_1^3 \exp(-C \delta_1 t) \norm{v-\tv}^2_{L^2(\R)} ,\]
we have
\begin{align*}
|\dot{X}_1\mathcal{Y}_{11}| &\leq \frac{\delta_1}{4}|\dot{X}_1|^2+\frac{C}{\delta_1}|\mathcal{Y}_{11}|^2 \leq \frac{\delta_1}{4}|\dot{X}_1|^2+C \mathcal{G}^S+C \delta_1^2 \exp(-C \delta_1 t)  \norm{v-\tv}^2_{L^2(\R)}.
\end{align*}
Likewise, we estimate $\dot{X}_2\mathcal{Y}_{21}$ as
\[|\dot{X}_2\mathcal{Y}_{21}| \leq \frac{\delta_2}{4}|\dot{X}_2|^2+C \mathcal{G}^S+C \delta_2^2 \exp(-C \delta_2 t)  \norm{v-\tv}^2_{L^2(\R)}.\]

\noindent $\bullet$ (Estimate of $\mathcal{Y}_{i2}$): We use a similar estimate as in \eqref{h-est} to obtain
\begin{align*}
|u-\tu| \leq |h-\widetilde{h}|+C(|(v-\tv)_x|+|\tv_x||v-\tv|+|(\tv_1)^{X_1}_x||\tv-\tv^{X_1}_1|+|(\tv_2)^{X_2}_x||\tv-\tv^{X_2}_2|).
\end{align*}
On the other hand, since
\[(p(v)-p(\tv))_x = p'(v)(v-\tv)_x+\tv_x(p'(v)-p'(\tv)),\]
we have
\[|(v-\tv)_x| \leq C|(p(v)-p(\tv))_x|+C|\tv_x||v-\tv|.\]
Then, using the localizations $\phi_i$,  and then the good term $G_1$ of  Lemma \ref{lem:rel-ent}, and Lemma \ref{lem:shock-interact-1} and  Lemma \ref{lem:shock-interact-2} with \eqref{est-rel-2},  we have
\begin{align*}
|\mathcal{Y}_{12}| &\leq C \int_\R |(\tv_1)^{X_1}_x| \bigg( \left| h-\widetilde{h}-\frac{p(v)-p(\tv)}{\sigma_1} \right| +|p(v)-p(\tv)|\\
&\hspace{3cm} + |(p(v)-p(v))_x|+|\tv_x||v-\tv|+|(\tv_1)^{X_1}_x||\tv-\tv^{X_1}_1|+|(\tv_2)^{X_2}_x||\tv-\tv^{X_1}_2| \bigg) \,dx\\
&\leq C \left( \frac{\delta_1}{\sqrt{\lambda}} \sqrt{{G}_{1}}+\sqrt{\delta_1}\sqrt{\mathcal{G}^S}+\int_\R |(\tv_1)^{X_1}_x||\phi_2(p(v)-p(\tv))| \,dx+ \delta_1 \sqrt{D}+\delta_1\delta_2\exp(-C \min(\delta_1,\delta_2)t)\right).
\end{align*}
This yields
\begin{align*}
|\dot{X_1}\mathcal{Y}_{12}| &\leq \frac{\delta_1}{4}|\dot{X}_1|^2+C\frac{\delta_1}{\lambda}{G}_1+C\mathcal{G}^S+C \delta_1 D\\
&\quad +C\delta_1^2 \exp(-C \delta_1 t)\norm{v-\tv}^2_{L^2(\R)}+C \delta_1 \delta_2^2 \exp(-C \min(\delta_1,\delta_2)t).
\end{align*}
A similar estimate holds for $|\dot{X}_2\mathcal{Y}_{22}|$.\\

\noindent $\bullet$ (Estimate of $\mathcal{I}_2$): We use the estimate \eqref{est-rel-3}$_1$ and the localization,  and then use Lemma \ref{lem:shock-interact-2}  with \eqref{est-rel-2} to obtain 
\begin{align*}
|\mathcal{I}_{21}|&\leq C \mathcal{G}^S +C\int_\R |(\tv_1)^{X_1}_x| \phi_2 |p(v)-p(\tv)|^2 \, dx\\
& \leq C \mathcal{G}^S+C\sqrt{\int_\R |(v_1)^{X_1}_x \phi_2|^2 \,dx}\sqrt{\int_\R |p(v)-p(\tv)|^4 \,dx}\\
&\leq  C \mathcal{G}^S + C \delta_1^{3/2} \exp(-C \delta_1 t) \norm{v-\tv}_{L^\infty(\R)} \norm{v-\tv}_{L^2(\R)},
\end{align*}
and the same estimate holds for $\mathcal{I}_{22}$.\\

\noindent $\bullet$ (Estimate of $\mathcal{I}_4$): To control $\mathcal{I}_4$, we will use the good term $\mathbf{D}_1$ generated by $\mathcal{I}_{3}$.
We use the localization,  and then use Lemma \ref{lem:shock-interact-2} with \eqref{est-rel-2} to derive
\begin{align*}
|\mathcal{I}_4| &\leq \int_\R |(u-\tu)_x||v-\tv|(|(\tu_1)^{X_1}_x|+|(\tu_2)^{X_2}_x|) \,dx\\
& \leq \frac{1}{4}\mathbf{D}_1+C(\delta_1 +\delta_2) \mathcal{G}^S\\
 &\quad +C\delta_1\int_\R |(\tv_1)^{X_1}_x| \phi_2 |p(v)-p(\tv)|^2 \,  dx+C\delta_2\int_\R |(\tv_2)^{X_2}_x| \phi_1 |p(v)-p(\tv)|^2 \,  dx\\
 & \leq \frac{1}{4}\mathbf{D}_1+C(\delta_1 +\delta_2)  \mathcal{G}^S +C\sum_{i=1}^2 \delta_i^{5/2} \exp(-C \delta_i t) \norm{v-\tv}_{L^\infty(\R)} \norm{v-\tv}_{L^2(\R)} .
\end{align*}

\noindent $\bullet$ (Estimate of $\mathcal{I}_6$) First, we note that the following estimate holds:
\begin{align*}
& \left| \sum_{i=1}^2 \left(\frac{(\tu_i)^{X_i}_x}{\tv_i^{X_i}}\right)_x-\left(\frac{\tu_x}{\tv}\right)_x \right|\\
&\quad \leq C \left(   (|(\tu_1)^{X_1}_{xx}|+|(\tu_1)^{X_1}_x| |(\tv_1)^{X_1}_x|)|\tv-\tv_1^{X_1}|+(|(\tu_2)^{X_2}_{xx}|+|(\tu_2)^{X_2}_x| |(\tv_2)^{X_2}_x|)|\tv-\tv_2^{X_2}| \right.\\
&\qquad + \left.  |(\tu_1)^{X_1}_x||(\tv_2)^{X_2}_x|+|(\tu_2)^{X_2}_x||(\tv_1)^{X_1}_x|\right)\\
&\quad \leq C \left(   (|(\tv_1)^{X_1}_{xx}|+|(\tv_1)^{X_1}_x|^2)|\tv-\tv_1^{X_1}|+(|(\tv_2)^{X_2}_{xx}|+|(\tv_2)^{X_2}_x|^2 )|\tv-\tv_2^{X_2}| +|(\tv_1)^{X_1}_x| |(\tv_2)^{X_2}_x| \right).
\end{align*}
Thanks to \eqref{second-order} in Lemma \ref{shock-est}, we have $|(\tv_i)^{X_i}_{xx}|,|(\tv_i)^{X_i}_x|^2\ll |(\tv_i)^{X_i}_x|$, and therefore, we estimate $\mathcal{I}_{6}$ as
\begin{align*}
\mathcal{I}_6 &\leq C\int_\R |u-\tu|( |(\tv_1)^{X_1}_x||\tv-\tv_1^{X_1}|+|(\tv_2)^{X_2}_x| |\tv-\tv_2^{X_2}|+ |(\tv_1)^{X_1}_x| |(\tv_2)^{X_2}_x|) \,  dx\\
&\leq C \varepsilon_1 \norm{|(\tv_1)^{X_1}_x||\tv-\tv_1^{X_1}|+|(\tv_2)^{X_2}_x| |\tv-\tv_2^{X_2}|+ |(\tv_1)^{X_1}_x| |(\tv_2)^{X_2}_x|}_{L^2},
\end{align*}
where we used a smallness of the perturbation \eqref{perturbation_small}.\\

\noindent Combining all the estimates and using  \eqref{perturbation_small} with Sobolev embedding,  we  have
\begin{align}
\begin{aligned}\label{E-3}
&\frac{d}{dt} \int_\R \eta(U(t,x)| \tU(t,x)) \, dx + \frac{1}{2} \mathbf{D}_1 \\
&\quad \leq \sum_{i=1}^2 \frac{\delta_i}{2} |\dot{X}_i|^2+C \sum_{i=1}^2 \frac{\delta_i}{\lambda}  G_1 +c_1 \mathcal{G}^S+C \sum_{i=1}^2 \delta_i D\\
&\qquad +C\varepsilon_1 \sum_{i=1}^2 \delta_i^{3/2} \exp(-C \delta_i t) + C\sum_{i=1}^2 \frac{\delta_1^2 \delta_2^2}{\delta_i} \exp(-C \min(\delta_1,\delta_2)t)\\
&\qquad +C \varepsilon_1  \norm{|(\tv_1)^{X_1}_x||\tv-\tv_1^{X_1}|+|(\tv_2)^{X_2}_x| |\tv-\tv_2^{X_2}|+ |(\tv_1)^{X_1}_x| |(\tv_2)^{X_2}_x|}_{L^2}.
\end{aligned}
\end{align}
On the other hand, we use the estimates
\[\|(\tv_i)^{X_i}_x\|_{L^\infty}\le \delta_i^2,\quad \|\tv-\tv_1^{X_1}\|_{L^\infty}\le \delta_2,\quad \|\tv-\tv_2^{X_2}\|_{L^\infty}\le \delta_1,\]
and Lemma \ref{lem:shock-interact-1} to derive the following inequalities :
\begin{align}
\begin{aligned}\label{L2-est}
\norm{|(\tv_i)^{X_i}_x|\tv-\tv_i^{X_i}|}_{L^2} &\leq C \delta_i ^{1/2} \delta_1\delta_2 \exp(-C \min(\delta_1,\delta_2) t),\\
\norm{|(\tv_1)^{X_1}_x| |(\tv_2)^{X_2}_x|}_{L^2} &\leq \delta_1^{3/2} \delta_2^{3/2} \exp(-C \min (\delta_1,\delta_2) t).
\end{aligned}
\end{align}
Integrating  \eqref{E-3} over $[0,t]$ for any $t \leq T$,  and using \eqref{L2-est},  we have
\begin{equation}\label{temp1}
\begin{aligned}
&\int_\R \left( \frac{|u-\tu|^2}{2}+Q(v|\tv) \right) \,  dx+\frac{1}{2} \int_0^t \mathbf{D}_1(U) \,  ds\\
&\quad \leq C\int_\R \left( \frac{|u_0-\tu(0,x)|^2}{2}+Q(v_0|\tv(0,x))\right) \,  dx\\
&\qquad +\int_0^t \left( \sum_{i=1}^2 \frac{\delta_i}{2} |\dot{X}_i|^2+C \sum_{i=1}^2 \frac{\delta_i}{\lambda} G_1+c_1 \mathcal{G}^S+C \sum_{i=1}^2 \delta_i D \right) \,  ds +C \varepsilon_1  \max(\delta_1,\delta_2)^{1/2}.
\end{aligned}
\end{equation}
Therefore,  multiplying \eqref{temp1} by the constant $\frac{1}{2\max (1,c_1)}$,  and then adding the result to \eqref{est-rel-ent},  together with the smallness of $\delta_i/\lambda,\delta_i,\varepsilon_1$,  we have
\begin{equation}\label{temp2}
\begin{aligned}
&\norm{v-\tv}_{L^2 (\R)}^2+\Vert h-\widetilde{h} \Vert_{L^2(\R)}^2+\norm{u-\tu}_{L^2(\R)}^2+\int_0^t \sum_{i=1}^2 \delta_i |\dot{X}_i|^2 \,ds+\int_0^t \left( \mathcal{G}^S+D+\mathbf{D}_1 \right) \,ds\\
&\leq C \left( \norm{v_0-\tv(0,\cdot)}_{L^2(\R)}^2 +\Vert (h-\widetilde{h})(0,\cdot)\Vert_{L^2(\R)}^2+\norm{u_0-\tu(0,\cdot)}_{L^2(\R)}^2 \right)+C \delta_0^{1/2}
\end{aligned}
\end{equation}
where we have used that 
\[C^{-1}|v-\tv|^2 \leq Q(v|\tv)\leq C |v-\tv|^2.\]
To complete the proof,  we need to control $\|(v-\tv)_x\|_{L^2(\R)}$ and $\|(h-\widetilde{h})(0,\cdot)\|_{L^2(\R)}$. Specifically, we will show that
\begin{equation} \label{temp 3}
\norm{(v-\tv)_x}_{L^2(\R)}^2 \leq C \left( \Vert h-\widetilde{h} \Vert_{L^2(\R)}^2+\norm{u-\tu}_{L^2(\R)}^2+\norm{v-\tv}_{L^2(\R)}^2+ \delta_1^2+\delta_2^2 \right)
\end{equation}
and
\begin{equation} \label{temp 4}
\Vert (h-\widetilde{h})(0,\cdot) \Vert_{L^2(\R)}^2 \leq C\left(\norm{v_0-\tv(0,\cdot)}_{H^1(\R)}^2+\norm{u_0-\tu(0,\cdot)}_{L^2(\R)}^2 +\delta_1^3 \delta_2^2+ \delta_1^2 \delta_2^3 \right).
\end{equation}
Using the definition of $h$ and $\widetilde{h}$,  we observe that
\[
\begin{aligned}
(u-\tu)-(h-\widetilde{h})&=(\ln v- \ln \tv_1^{X_1} -\ln \tv_2^{X_2})_x\\
&=\frac{(v-\tv)_x}{v}+\frac{\tv-v}{v \tv} (\tv)_x-\frac{\tv_2^{X_2}}{\tv \tv_1^{X_1}} (\tv_1)^{X_1}_x-\frac{\tv_1^{X_1}}{\tv \tv_2^{X_2}}(\tv_2)^{X_2}_x
\end{aligned} \]
which yields 
\[
\begin{aligned}
(v-\tv)_x=v(u-\tu)-v(h-\widetilde{h})-\frac{(\tv)_x}{\tv}(v-\tv)+\frac{v \tv_2^{X_2}}{\tv \tv_1^{X_1}}(\tv_1)^{X_1}_x+\frac{v \tv_1^{X_1}}{\tv \tv_2^{X_2}}(\tv_2)^{X_2}_x.
\end{aligned}
\]
This, together with the fact that $\norm{(\tv_i)^{X_i}_x}_{L^2(\R)}^2 \leq C \delta_i^2 $ implies \eqref{temp 3}. On the other hand, it follows from the estimate \eqref{h-est} that 
\[\Vert (h-\widetilde{h})(0,\cdot)\Vert_{L^2(\R)}^2 \leq C \left(\norm{v_0-\tv(0,\cdot)}_{H^1(\R)}^2+\norm{u_0-\tu(0,\cdot)}_{L^2(\R)}^2+C \delta_1^3 \delta_2^2+C \delta_1^2 \delta_2^3\right),\]
which yields \eqref{temp 4}. Hence, combining \eqref{temp2},\eqref{temp 3}  and \eqref{temp 4} with $D_1\le C  \mathbf{D}_1$, we obtain the desired estimate \eqref{temp 0}.
\end{proof}

\subsection{First-order estimate}
We now present the following estimate for $\norm{u-\tu}_{H^1(\R)}$, which completes the proof of Proposition \ref{prop:H1-estimate}.

\begin{lemma} Under the hypotheses of Proposition 3.2,  there exists a positive constant $C>0$, independent of $\delta_1,\delta_2,\varepsilon_1,T$, such that for all $t\in (0,T]$,
\begin{equation}
\begin{aligned}
&\norm{v-\tv}_{H^1(\R)}^2+\norm{u-\tu}_{H^1(\R)}^2+\int_0^t \sum_{i=1}^2 \delta_i |\dot{X}_i|^2 \,ds +\int_0^t \left( \mathcal{G}^S+D+D_1+{D}_2 \right) \,ds\\
&\leq C \left( \norm{v_0-\tv(0,\cdot)}_{H^1(\R)}^2+\norm{u_0-\tu(0,\cdot)}_{H^1(\R)}^2 \right)+C \delta_0^{1/2},
\end{aligned}
\end{equation}
where $\mathcal{G}^S,D$ are as in Lemma \ref{lem:rel-ent},  and $D_1$ is as in \eqref{temp 0},  and
\[{D}_2(U):=\int_\R | (u-\tu)_{xx}|^2 \, dx.\]
\end{lemma}

\begin{proof}
Considering \eqref{temp 0}, we only need to estimate $\|\pa_x(u-\tu)\|_{L^2(\R)}$. For notational simplicity, we define $\psi:=u-\tu$. Then, $\psi$ satisfies
\begin{equation}\label{eq:psi}
	\psi_t-\sum_{i=1}^2 \dot{X}_i (\tu_i)^{X_i}_x+(p(v)-p(\tv))_x=\left( \frac{u_x}{v}-\frac{\tu_x}{\tv}\right)_x-E_2-E_3.
\end{equation}
We multiply \eqref{eq:psi} by $-\psi_{xx}$ and integrate over $\bbr$ to obtain
\[
\begin{aligned}
\frac{d}{dt} \int_\R \frac{|\psi_x|^2}{2}\,dx
&=-\sum_{i=1}^2 \dot{X}_i \int_\R (\tu_i)^{X_i}_x \psi_{xx} \, dx+\int_\R (p(v)-p(\tv))_x \psi_{xx} \, dx\\
&\quad -\int_\R \left( \frac{u_x}{v}-\frac{\tu_x}{\tv}\right)_x \psi_{xx} \,dx+\int_\R (E_2+E_3) \psi_{xx} \,dx\\
&=:J_1+J_2+J_3+J_4.
\end{aligned}
\]

\noindent $\bullet$ (Estimate of $J_1$ and $J_2$): We first define the good term
\[\mathbf{D}_2:=\int_\R \frac{1}{v} |\psi_{xx}|^2 \, dx,\]
and split $J_1$ as $J_{11}+J_{12}$ where
\[J_{1i}:=-\dot{X}_i\int_{\R}(\tu_i)_x^{X_i}\psi_{xx}\,d x.\]
Then, using Holder inequality and Lemma \ref{lem:shock-est},  and then Young's inequality,  we obtain
\[\begin{aligned}
|J_{1i}| &\leq  |\dot{X}_i| \sqrt{\int_\R |(\tu_i)^{X_i}_x |^2 \,dx}\sqrt{\int_\R |\psi_{xx}|^2 \,dx} \\
&\leq |\dot{X}_i| \delta_i^{3/2} \sqrt{ \int_\R |\psi_{xx}|^2 \,dx} \leq \frac{\delta_i}{2} | \dot{X}_i|^2+C \delta_i^2 \mathbf{D}_2 \leq \frac{\delta_i}{2}|\dot{X}_i|^2+\frac{1}{16}\mathbf{D}_2,
\end{aligned}\]
which yields
\[|J_1|\le\sum_{i=1}^2\frac{\delta_i}{2}|X_i|^2 + \frac{1}{8}\mathbf{D}_2.\]
Similarly,  using Young's inequality,  we have
\[|J_2| \leq \frac{1}{8} \mathbf{D}_2+C D.\]

\noindent $\bullet$ (Estimate of $J_3$): We estimate $J_3$ as
\[
\begin{aligned}
J_3&:=-\int_\R \frac{1}{v} |\psi_{xx}|^2 \,dx-\int_\R \left( \frac{1}{v} \right)_x \psi_x \psi_{xx} \,dx\\
&\quad -\int_\R \tu_{xx} \left( \frac{1}{v}-\frac{1}{\tv} \right)\psi_{xx} \,dx-\int_\R \tu_x \left( \frac{1}{v}-\frac{1}{\tv} \right)_x \psi_{xx} \,  dx\\
&=:-\mathbf{D}_2+J_{31}+J_{32}+J_{33},
 \end{aligned}
\]
from which we derive the good term $\mathbf{D}_2$. On the other hand, we use $(\frac{1}{v})_x\leq C |v_x| \leq C(|(v-\tv)_x|+|\tv_x|)$ and the interpolation inequality to derive
\[
\begin{aligned}
|J_{31}| &\leq \norm{(v-\tv)_x}_{L^2}\norm{\psi_x}_{L^\infty}\norm{\psi_{xx}}_{L^2}+\norm{\tv_x}_{L^\infty}\norm{\psi_x}_{L^2}\norm{\psi_{xx}}_{L^2}\\
&\leq C \varepsilon_1 \norm{\psi_x}_{L^2}^{1/2}\norm{\psi_{xx}}_{L^2}^{1/2}\norm{\psi_{xx}}_{L^2}+C(\delta_1+\delta_2)\norm{\psi_x}_{L^2}\norm{\psi_{xx}}_{L^2}\\
&\leq C(\varepsilon_1+\delta_1+\delta_2)(\norm{\psi_x}_{L^2}^2+\norm{\psi_{xx}}_{L^2}^2)\leq\frac{1}{8}\mathbf{D}_2+C(\varepsilon_1+\delta_1+\delta_2)D_1.
\end{aligned}
\]
Moreover, using $|(\tu_i)^{X_i}_{xx}|  \leq C |(\tu_i)^{X_i}_x|$ in Lemma \ref{lem:shock-est}, relation $\phi_1+\phi_2=1$,  and then the same estimate as before,  we have
\[
\begin{aligned}
|J_{32}|&\leq C \int_\R (|(\tu_1)^{X_1}_x|+|(\tu_2)^{X_2}_x|)|v-\tv||\psi_{xx}| \,dx\\
&\leq \frac{1}{8}\mathbf{D}_2+C\delta_1 \mathcal{G}^S+C \delta_2 \mathcal{G}^S\\
&\quad +C\delta_1\int_\R |(\tv_1)^{X_1}_x| \phi_2 |p(v)-p(\tv)|^2 \,dx+C \delta_2 \int_\R |(\tv_2)^{X_2}_x| \phi_1|p(v)-p(\tv)|^2 \,dx\\
&\leq \frac{1}{8}\mathbf{D}_2+C\delta_1 \mathcal{G}^S+C\delta_2 \mathcal{G}^S+C \varepsilon_1^2 \sum_{i=1}^2  \delta_i^{5/2} \exp(-C \delta_i t).
\end{aligned}
\]
Finally, we estimate $J_{33}$ as
\[
\begin{aligned}
|J_{33}| &\leq C \int_\R (|(\tu_1)^{X_1}_x|+|(\tu_2)^{X_2}_x|)(|v-\tv|+|(v-\tv)_x|) |\psi_{xx}| \,dx\\
&\leq \frac{1}{8}\mathbf{D}_2+C \int_\R \left( |(\tv_1)^{X_1}_x|^2+|(\tv_2)^{X_2}_x|^2 \right) \left( |v-\tv|^2+|(v-\tv)_x|^2\right) dx\\
&\leq \frac{1}{8}\mathbf{D}_2+C(\delta_1+\delta_2)\left(D+\mathcal{G}^S+\varepsilon_1^2 \sum_{i=1}^2  \delta_i^{5/2} \exp(-C \delta_i t) \right).
\end{aligned}\]
\noindent $\bullet$ (Estimate of $J_4$): We control $J_4$ as
\[
\begin{aligned}
|J_4| &\leq C \norm{\psi_{xx}}_{L^2}\norm{ (|(\tv_1)^{X_1}_{xx}|+|(\tv_1)^{X_1}_x|^2)|\tv-\tv_1^{X_1}|+(|(\tv_2)^{X_2}_{xx}|+|(\tv_2)^{X_2}_x|^2 )|\tv-\tv_2^{X_2}| +|(\tv_1)^{X_1}_x| |(\tv_2)^{X_2}_x|}_{L^2}\\
&\leq \frac{1}{8}\mathbf{D}_2+C\norm{|(\tv_1)^{X_1}_x||\tv-\tv_1^{X_1}|+|(\tv_2)^{X_2}_x||\tv-\tv_2^{X_2}|+|(\tv_1)^{X_1}_x||(\tv_2)^{X_2}_x|}_{L^2}^2.
\end{aligned}\]
Therefore, combining the estimates for $J_i$, there exists a positive constant $c_2>0$ such that
\begin{align*}
\frac{d}{dt} \int_\R \frac{|\psi_x|^2}{2} \,dx &\leq -\frac{1}{4} \mathbf{D}_2+\sum_{i=1}^2 \frac{\delta_i}{2}|\dot{X}_i|^2+c_2 D+C(\varepsilon_1+\delta_1+\delta_2)(\mathcal{G}^S+D_1)+C \varepsilon_1^2  \sum_{i=1}^2  \delta_i^{5/2} \exp(-C \delta_i t)\\
&\quad +C\norm{|(\tv_1)^{X_1}_x||\tv-\tv_1^{X_1}|+|(\tv_2)^{X_2}_x||\tv-\tv_2^{X_2}|+|(\tv_1)^{X_1}_x||(\tv_2)^{X_2}_x|}_{L^2}^2.
\end{align*}
Integrating the above estimate over $[0,t]$ for any $t \leq T$, and using \eqref{L2-est}, we obtain
\begin{align*}
\int_\R \frac{|(u-\tu)_x|^2}{2} \,dx &\leq \int_\R \frac{|(u_0-\tu(0,x))_x|^2}{2} \,dx+\int_0^t \left(-\frac{1}{4} \mathbf{D}_2+\sum_{i=1}^2 |\dot{X}_i|^2 \right) ds\\
&\quad +\int_0^t \left( c_2 D+C(\varepsilon_1+\delta_1+\delta_2)(\mathcal{G}^S+D_1) \right) ds+C\varepsilon_1 \max (\delta_1,\delta_2)^{1/2}.
\end{align*}
Multiplying the above inequality by the constant $\frac{1}{2 \max(1,c_2)}$ and then adding the result \eqref{temp 0},  together with the smallness of $\varepsilon_1,\delta_1,\delta_2$,  we have
\begin{align*}
&\norm{v-\tv}_{H^1(\R)}^2+\norm{u-\tu}_{H^1(\R)}^2+\int_0^t \left( \sum_{i=1}^2 \delta_i |\dot{X}_i|^2+ \mathcal{G}^S+D+\mathcal{D}_1+\mathbf{D}_2 \right)ds\\
&\leq C \left( \norm{v_0-\tv(0,\cdot)}_{H^1(\R)}^2+\norm{u_0-\tu(0,\cdot)}_{H^1(\R)}^2 \right)+C  \delta_0^{1/2},
\end{align*}
which is the desired estimate \eqref{C-3}. This completes the proof of Proposition \ref{prop:H1-estimate}.
\end{proof}

\begin{appendix}
\setcounter{equation}{0}
\section{Proof of Lemma \ref{lem:shock-interact-1}}  \label{app-1}
In this appendix, we present the detailed proof of Lemma \ref{lem:shock-interact-1}. We only consider the case of $i=1$, since the other case can be shown in the same manner. It follows from \eqref{shock-est} that 
\begin{equation}\label{eqA-1}
|(\tv_1)^{X_1}_x| = |\tv'_1(x-\sigma_1t-X_1(t))|\le C\delta_1^2\exp(-C\delta_1|x-\sigma_1t-X_1(t)|),\quad \forall x\in \bbr,\quad t>0
\end{equation}
and
\begin{equation}\label{eqA-2}
|\tv^{X_1,X_2}-\tv_1^{X_1}|=|\tv_2^{X_2}-v_m|\le \begin{cases}
C\delta_2\exp(-C\delta_2|x-\sigma_2t-X_2(t)|),\quad&\mbox{if}\quad x\le \sigma_2t+X_2(t),\\
C\delta_2,\quad &\mbox{if}\quad x\ge \sigma_2t+X_2(t).
\end{cases}
\end{equation}
Recall that \eqref{sepx12} implies $X_2(t)+\sigma_2t\ge\frac{\sigma_2}{2}t>0$. Therefore, we combine \eqref{eqA-1} and \eqref{eqA-2} to derive
\begin{equation}\label{eqA-3}
|(\tv_1)_x^{X_1}||\tv^{X_1,X_2}-\tv_1^{X_1}|\le \begin{cases}
C\delta_1^2\delta_2\exp(-C\delta_2|x-\sigma_2t-X_2(t)|)\quad&\mbox{if}\quad x\le0,\\
C\delta_1^2\delta_2\exp(-C\delta_1|x-\sigma_1t-X_1(t)|),\quad&\mbox{if}\quad x\ge0,
\end{cases}
\end{equation}
and
\begin{equation}\label{eqA-4}
|(\tv_1)_x^{X_1}|^{1/2}|\tv^{X_1,X_2}-\tv_1^{X_1}|\le \begin{cases}
C\delta_1\delta_2\exp(-C\delta_2|x-\sigma_2t-X_2(t)|)\quad&\mbox{if}\quad x\le0,\\
C\delta_1\delta_2\exp(-C\delta_1|x-\sigma_1t-X_1(t)|),\quad&\mbox{if}\quad x\ge0.
\end{cases}
\end{equation}
Again, we note that \eqref{sepx12} implies
\begin{align}
\begin{aligned}\label{eqA-5}
&x-\sigma_2t-X_2(t)\le x-\frac{\sigma_2}{2}t\le-\frac{\sigma_2}{2}t<0,\quad \mbox{if}\quad x\le0,\\
&x-\sigma_1t-X_1(t)\ge x-\frac{\sigma_1}{2}t\ge-\frac{\sigma_1}{2}t>0,\quad \mbox{if}\quad x\ge0,
\end{aligned}
\end{align}
which combined with \eqref{eqA-3} yields the first desired estimate:
\[|(\tv_1)_x^{X_1}||\tv^{X_1,X_2}-\tv_1^{X_1}|\le
C\delta_1^2\delta_2\exp(-C\min\{\delta_1,\delta_2\}t)\quad\mbox{if}\quad\forall x\in\bbr,\quad t>0.
\]
Similarly, we combine \eqref{eqA-4} and \eqref{eqA-5} to obtain
\[|(\tv_1)_x^{X_1}|^{1/2}|\tv^{X_1,X_2}-\tv_1^{X_1}|\le
C\delta_1\delta_2\exp(-C\min\{\delta_1,\delta_2\}t)\quad\mbox{if}\quad\forall x\in\bbr,\quad t>0.
\]
On the other hand, since \eqref{shock-est} implies
\begin{equation}\label{eqA-6}
\int_{\bbr} |(\tv_1)_x^{X_1}|^{1/2}\,d x\le C,
\end{equation}
we obtain the second estimate. Finally, we again use \eqref{shock-est} and \eqref{eqA-5} to find
\[|(\tv_1)_x^{X_1}|^{1/2}|(\tv_2)_x^{X_2}|\le C\delta_1\delta_2^2\exp(-C\min\{\delta_1,\delta_2\}t),\quad \forall x\in\bbr,\quad t>0.\]
We combine the above estimate with \eqref{eqA-6} to derive the last desired estimate.

\section{Proof of Lemma \ref{lem:shock-interact-2}}  \label{app-2}
\setcounter{equation}{0}
In this appendix, we provide the proof of Lemma \ref{lem:shock-interact-2}. Again, we only focus on the case of $\tv_1$. It follows from \eqref{eqA-1} that
\begin{equation}\label{eqB-1}
\phi_2|(\tv_1)_x^{X_1}|^{1/2}\le C\phi_2\delta_1 \exp(-C\delta_1|x-\sigma_1t-X_1(t)|),\quad\forall x\in\bbr,\quad t>0,
\end{equation}
and
\begin{equation}\label{eqB-2}
\phi_2|(\tv_1)_x^{X_1}|\le C\phi_2\delta^2_1 \exp(-C\delta_1|x-\sigma_1t-X_1(t)|),\quad\forall x\in\bbr,\quad t>0.
\end{equation}
We note that the support of $\phi_2(t,\cdot)$ is $\{x~:~x\ge (X_1(t)+\sigma_1t)/2\}$, on which the following estimate holds:
\[x-(\sigma_1t+X_1(t))\ge -\frac{X_1(t)+\sigma_1t}{2}\ge -\frac{\sigma_1t}{4}>0.\]
Since $0\le\phi_2\le 1$, we use \eqref{eqB-2} to derive
\[\phi_2|(\tv_1)_x^{X_1}|\le C\delta^2_1\exp(-C\delta_1t),\]
which yields the desired first estimate. Similarly, we use \eqref{eqB-1} to obtain
\[\phi_2|(\tv_1)_x^{X_1}|^{1/2}\le C\delta_1\exp(-C\delta_1t).\]
Combining the above estimate with \eqref{eqA-6}, we derive the second estimate. This complete the proof of Lemma \ref{lem:shock-interact-2}.
\end{appendix}
\bibliographystyle{amsplain}
\bibliography{reference}

\end{document}